\documentclass[11pt]{amsart}

\usepackage{amscd,amsfonts,mathrsfs,amsthm,amssymb}
\usepackage{comment}
\usepackage{mathtools}
\usepackage{enumerate}
\usepackage{multicol}        			
\usepackage{color}           			
\usepackage{array}
\usepackage{ae}
\usepackage{accents}
\usepackage{epsfig}
\usepackage[e]{esvect}

\usepackage{yfonts}

\usepackage{empheq}
\usepackage[latin1]{inputenc}

\usepackage{hyperref}
\usepackage[alphabetic, msc-links, backrefs]{amsrefs}

\definecolor{r}{rgb}{.9,0.1,.3}

\definecolor{verdeosc}{rgb}{0,0.6,0}

\newcommand{\Ff}{\mathcal F}

 \newcommand{\Dd}{\mathcal{D}}
 \newcommand{\Ee}{\mathcal{E}}

  \newcommand{\Ss}{\mathcal S}
 
 \newcommand{\Qint}{Q^\textnormal{int}}
\newcommand{\Qb}{Q^{\partial}}
 \newcommand{\RR}{\mathbf{R}}  
 \newcommand{\Ww}{\mathcal{W}}
 \newcommand{\ZZ}{\mathbf{Z}}  
 \newcommand{\BB}{\mathbf{B}}  
 \newcommand{\CC}{\mathbf{C}}  

 \newcommand{\eps}{\epsilon}
 \newcommand{\Tan}{\operatorname{Tan}}
 \newcommand{\Tt}{\mathcal{T}}

\newcommand{\pdf}[2]{\frac{\partial #1}{\partial #2}}

\newtheorem*{theorem*}{Theorem}
\newtheorem*{lemma*}{Lemma}
\newtheorem{theorem}{Theorem}

\newtheorem{lemma}[theorem]{Lemma}

\newtheorem{corollary}[theorem]{Corollary}
\newtheorem{proposition}[theorem]{Proposition}
\newtheorem{definition}[theorem]{Definition}
\newtheorem{claim}[theorem]{Claim}
\newtheorem*{claim*}{Claim}
\theoremstyle{definition}
\newtheorem{remark}[theorem]{Remark}

\def\pproof#1{\@ifnextchar[\opargproof
{\opargproof[\it Proof of #1.]}}
\def\opargproof[#1]{\par\noindent {\bf #1 }}

\makeatother

\title{Morse-Rad\'o Theory for Minimal Surfaces}
\author[D. Hoffman]{\textsc{D. Hoffman}}

\address{David Hoffman\newline
 Department of Mathematics\newline
 Stanford University \newline
   Stanford, CA 94305, USA\newline
{\sl E-mail address:} {\bf dhoffman@stanford.edu}}

\author[F. Martin]{\textsc{F. Martín}}

\address{Francisco Martín\newline
Departamento de Geometría y Topología  \newline
Instituto de Matemáticas IMAG \newline
Universidad de Granada\newline
18071 Granada, Spain\newline
{\sl E-mail address:} {\bf fmartin@ugr.es}
}
\author[B. White]{\textsc{B. White}}

\address{Brian White\newline
Department of Mathematics \newline
 Stanford University \newline 
  Stanford, CA 94305, USA\newline
{\sl E-mail address:} {\bf bcwhite@stanford.edu}
}

\subjclass[2010]{Primary 53A10, Secondary 49Q05, 53C42}
\keywords{Minimal surfaces, critical points}
\thanks{F. Martín  was partially supported by the MICINN grant PID2020-116126-I00, by the IMAG--Maria de Maeztu grant CEX2020-001105-M / AEI / 10.13039/501100011033 and  by the Regional Government of Andalusia and ERDEF grant PY20-01391.}

\date{October 19, 2022}                                   

\begin{document}

\begin{abstract}
For a class of functions (called minimal Rad\'o functions) that arise naturally in minimal surface
theory, we bound the number of interior critical points (counting multiplicity) in
terms of the boundary data and the Euler characteristic of the domain of the function.
\end{abstract}

\maketitle

\tableofcontents

\section{Introduction}

Consider the following facts from the classical theory of minimal surfaces:
\begin{enumerate}[\upshape(1)]
\item\label{Rado-item} If $M$ is a compact minimal disk in $\RR^n$, if $F:\RR^n\to \RR$ is linear, and if 
  $F^{-1}(c)\cap M$ contains an interior critical point of $F|M$ with multiplicity $k$, then
  $F^{-1}(c)\cap \partial M$ contains at least $2 (k+1)$ points.
  (See~\cite{Rado1930}*{p.~794,(c)}).
\item\label{schneider-item}
 If $M$ is a minimal disk in $\RR^n$, if $F:\RR^n\to \RR$ is linear,
 and if $F:\partial M \to \RR$ has at most $k$ local minima, then $F|M$ has at most
 $k-1$ interior critical points, counting multiplicity.
 (See~\cite{schneider}*{Lemma~2}.)
 \end{enumerate}
 
These facts are powerful tools in minimal surface theory.
For instance, Rad\'o~\cite{rado} used~\eqref{Rado-item} to prove that if the boundary of  minimal disk in $\RR^n$ 
projects homeomorphically to the boundary of a convex region in a plane, then the
interior of the disk is a smooth graph over that region.  (Actually, Rad\'o stated the theorem only for $n=3$, but Osserman~\cite{osserman-book}*{Theorem~7.2} pointed out that Rad\'o's proof works for any $n$.)  Rad\'o also showed that if the boundary of a minimal disk in $\RR^n$ projects homeomorphically onto the boundary of a planar star-shaped region, then the interior of 
the disk has no branch points~\cite{Rado1930}*{p.~794}.
Finn and Osserman~\cite{FO} used an analog of~\eqref{Rado-item}
to prove a curvature estimate that implies Bernstein's Theorem (an entire
solution $u:\RR^2\to\RR$ of the minimal surface equation must be a plane.)
Schneider~\cite{schneider}
 used~\eqref{schneider-item} to show that for a minimal
disk in Euclidean space, the sum of the
orders of the interior branch points is bounded by 
\[
 \frac{\kappa}{2\pi} -1,
\]
where $\kappa$ is the total curvature of the boundary.
More recently,~\eqref{schneider-item} was used in the variational existence proof of genus-one helicoids
in $\RR^3$~\cite{hoffman-white}.

In this paper, we sharpen~\eqref{Rado-item} and~\eqref{schneider-item} and extend them to minimal surfaces of arbitrary genus in 
Riemannian manifolds.   
See Theorem~\ref{slice-theorem} for the generalization of~\eqref{Rado-item}.
Generalizing~\eqref{schneider-item}, we show:

\begin{theorem} \label{th-one}
Suppose that $M$ is a compact minimal surface with boundary in a Riemannian manifold $N$.
Suppose that $F:N\to\RR$ is a continuous function  such that
\begin{enumerate}[\upshape(1)]
\item\label{hypothesis-1} if $\dim N=3$, the level sets of $F$ are minimal surfaces, and 
\item\label{hypothesis-2} if $\dim N>3$, the the level sets of $F$ are totally geodesic.
\item\label{hypothesis-3} for each $t$, $\{F=t\}$ is in the closure of $\{F>t\}$ and of $\{F<t\}$.
\end{enumerate}
Suppose also that $F$ is nonconstant on each connected component of $M$,
and that the set $Q$ of local minima of $F|\partial M$ is finite.
Then the number $\mathsf{N}(F|M)$ of interior critical points of $F|M$ (counting multiplicity) 
and the number $s^\partial(F)$ of boundary saddle points of $F|M$ (counting multiplicity) satisfy
\begin{equation*}
  \mathsf{N}(F|M) + s^\partial(F) =  |Q| - \chi(M),
\end{equation*}
where 
$\chi(M)$ is the Euler characteristic of $M$ 
and where $|Q|$ is the number of elements in the set $Q$.
\end{theorem}

(Theorem~\ref{th-one} is a special case of Theorem~\ref{restated-theorem}; Theorem~\ref{minimal-surface-theorem}
 and Remarks~\ref{higher-codimension-remark} and~\ref{minimal-surface-theorem-remark} show that the hypotheses
of Theorem~\ref{th-one} imply the hypotheses of Theorem~\ref{restated-theorem}.  
Theorem~\ref{th-one} is also true for branched minimal surfaces; see~\S\ref{branch-section}.) 

A continuous function whose level sets form a foliation and that satisfies hypothesis~\eqref{hypothesis-3} of Theorem~\ref{th-one} is called a foliation function.  If the leaves are minimal, it is called a minimal foliation function, and
if the leaves are totally geodesic, it is called 
  a totally geodesic foliation function.

In Theorem~\ref{th-one}, ``interior critical point of $F|M$'' means ``interior point $p$ of tangency of $M$ and the level set $\{F=F(p)\}$'',
and the multiplicity of such a critical point is the order of contact of $M$ and $\{F=F(p)\}$.
Boundary saddle points and their multiplicities are defined in Definition~\ref{boundary-saddle-definition}.

It would be natural in Theorem~\ref{th-one} to assume that $F$ is $C^1$ (or even smooth) with nowhere vanishing gradient.
However, that assumption would be undesirable for the following reason.
Consider a minimal foliation $\Ff$ of a Riemannian $3$-manifold.  Of course the leaves
are smooth.   At least locally, the foliation can be given as the level sets of a continuous function $F$.
However, for some minimal foliations, there is no such function that is $C^1$ with nowhere vanishing gradient.
(A simple example from~\cite{solomon}*{\S1} is the minimal foliation of $\{(x,y,z): x>0\}$
 consisting of the halfplanes $z= sx$ with $s\ge 0$
and the halfplanes $z=s$ with $s<0$.  If $F$ is a $C^1$ function whose level sets are the leaves, then $DF(x, y ,0) = 0$.)

For that reason, throughout the paper we work with functions that are only assumed to be continuous.

Theorem~\ref{th-one} provides an exact formula for $\mathsf{N}(F|M)$.
In many situations, a good upper bound for $\mathsf{N}(F|M)$ suffices.
Simply dropping the term $s^\partial(F)$ in Theorem~\ref{th-one} gives the bound 
\[
   \mathsf{N}(F|M) \le |Q| - \chi(M),
\]
which is often adequate.  Indeed, that gives Schneider's bound~\eqref{schneider-item}.
But one can get a better upper bound as follows.
Let $A$ be the set of local maxima and local minima of $F|\partial M$ that are not
local maxima or local minima of $F|M$.  Then $s^\partial(F) \ge |A|$ (where $|A|$ is the number of elements of $A$),
so from Theorem~\ref{th-one}, we deduce

\begin{corollary}\label{th-one-corollary}
Under the hypotheses of Theorem~\ref{th-one}, 
\begin{equation*}
   \mathsf{N}(F|M)  \le |Q| - \chi(M) - |A|.
\end{equation*}
\end{corollary}

See Theorem~\ref{restated-inequality-theorem}, which also specifies when equality holds in Corollary~\ref{th-one-corollary}.

\begin{remark}
 In practice, one sometimes encounters $F$ and $M$ that satisfy all but one of the hypotheses of Theorem~\ref{th-one}, namely the hypothesis that the set of local minima of $F|\partial M$ is finite. 
 In particular, that hypothesis will fail if $F$ is constant on one or more arcs of $\partial M$. One can handle such examples as follows. Suppose $F$ is not constant on any connected component of $\partial M$.
  Let $\tilde M$ be obtained from $M$ by identifying each arc of $\partial M$ on which $F$ is constant to a point. 
  Let $\tilde F$ be the function on $\tilde M$ corresponding to $F$  on $M$. If $\tilde F|\partial \tilde M$ has a finite 
   set $\tilde Q$ of local minima, then
\begin{align*}
\mathsf{N}(F|M)
&=
\mathsf{N}(\tilde F| \tilde M)  \\
&=  |\tilde Q| -\chi(M) - s^\partial (\tilde F) \\
&\le  |\tilde Q| - \chi(M) -  |\tilde A|,
\end{align*}
where $\tilde A$ is the set of local minima and local maxima of $\tilde F| \partial \tilde M$ that are not local minima or local maxima     
      of $\tilde F|M$. 
  These facts follow from Theorems~\ref{restated-theorem}, \ref{restated-inequality-theorem}, 
and~\ref{interval-theorem}.
\end{remark}

Special cases of Theorem~\ref{th-one} have been important tools for analyzing properly embedded
translators for mean curvature flow in $\RR^3$, in particular
for the classification of translating graphs in \cites{graphs, himw-correction},
the classification of semigraphical translators (such as the doubly-periodic
Scherk-type translators and the Nguyen singly-periodic translators) \cites{HMW1, HMW2},
the classification of low entropy translators~\cite{?},
and for the construction of families of non-rotationally invariant 
translating annuli (analogs of catenoids) \cite{HMW3}.

There is also a version of Theorem~\ref{th-one} for noncompact $M$:

\begin{theorem}\label{noncompact-intro-theorem}
Let $-\infty\le a < b \le \infty$.
In Theorem~\ref{th-one}, suppose the hypothesis that $M$ is compact is replaced by the 
hypotheses that $F:M\to (a,b)$ is proper,
 that $d_1(M):=\dim H_1(M;\ZZ_2)$ is finite, and that
the limit
\[
  \beta:=\lim_{t\to a, \, t>a} | (\partial M)\cap F^{-1}(t)| 
\]
exists and is finite.   
Then
\begin{equation*}
  \mathsf{N}(F|M) + s^\partial(F) = \frac12 \beta+ |Q| - \chi(M),
\end{equation*}
and therefore
\begin{equation*}
  \mathsf{N}(F|M) \le \frac12 \beta+ |Q| - \chi(M) - |A|.
\end{equation*}

\end{theorem}

Theorem~\ref{noncompact-intro-theorem} is a special case of Corollary~\ref{general-interval-corollary},
by virtue of Theorem~\ref{minimal-surface-theorem}, 
Remarks~\ref{higher-codimension-remark} and~\ref{minimal-surface-theorem-remark}),
and (for the inequality involving $|A|$) Proposition~\ref{inequality-proposition}.

Another useful fact about $\mathsf{N}(F|M)$ is that it depends lower semicontinuously on $F$ and on $M$ (even without assuming 
properness); see Theorem~\ref{lower-semicontinuity-corollary}.

The paper is organized as follows.  We define a class of functions on surfaces that
we call Rad\'o functions.  Roughly speaking, they are continuous functions whose level sets are locally 
either isolated points or (qualitatively) like the level sets of harmonic functions. (The isolated points occur
at strict local minima and at strict local maxima.)
We show that if $M$ is a minimal surface
in a smooth Riemannian manifold $N$ and if $F:N\rightarrow\RR$ is a continuous function satisfying 
  hypotheses~\eqref{hypothesis-1}, \eqref{hypothesis-2}, and~\eqref{hypothesis-3}
 of Theorem~\ref{th-one}, then $F$ is a Rad\'o function on the interior of $M$. 
Under mild hypotheses, it follows that $F$ is a Rad\'o function on all of $M$; see Theorem~\ref{boundary-is-rado-theorem}.
We then prove the various theorems bounding numbers of critical points for arbitrary Rad\'o functions.

\section{Rad\'o Functions}

\begin{definition}\label{rado-definition}
A continuous, real-valued function on a $2$-manifold $M$ is called a {\bf Rad\'o function}
provided each point $p\in M$ has a neighborhood $U$ such that
\begin{enumerate}[\upshape(1)]
\item $U\cap \{F=F(p)\}$ consists of a finite collection $C_1,\dots, C_v$ of embedded arcs.
\item Each $C_i$ joins the point $p$ to a point in $\partial U$.
\item $C_i\cap C_j=\{p\}$ for $i\ne j$.
\item\label{crossing-item} Each $C_i$ is in the closure of $\{F>F(p)\}$ and in the closure of $\{F<F(p)\}$.
\item If $p\in \partial M$, we also require that each $C_i\setminus\{p\}$ is contained in the interior of $M$.
\end{enumerate}
The number $v=v(F,p)$ is called the {\bf valence} of $p$.
\end{definition}

We call these functions Rad\'o functions because Rad\'o
 observed~\cite{rado}*{III.6} that some important properties of harmonic functions on surfaces are shared by functions similar to those in Definition~\ref{rado-definition} (provided there are no points of valence $0$.)
 
 Note that for a Rad\'o function $F:M\to\RR$,
 \begin{enumerate}
 \item The points of valence $0$ are the local minima and local maxima of $F$.
 \item Each local maximum (local minimum) of a Rad\'o function is a strict local maximum (local minimum).
 \item For each $t$, the set $(\partial M)\cap F^{-1}(t)$ is discrete.  Indeed, if $p$ and $U$ are as in 
    Definition~\ref{rado-definition} and if $p\in\partial M$, then $(\partial M)\cap U \cap \{F=F(p)\}$ consists only of the point $p$.
\end{enumerate}

The following lemma is an immediate consequence of Definition~\ref{rado-definition}; see Figure~\ref{fig:1}.

\begin{lemma}\label{parity-lemma}
Suppose that $F:M\to\RR$ is a Rad\'o function.
If $p$ is an interior point, then $v(F,p)$ is even.
If $p$ is a boundary point, then $v(F,p)$ is even if and only if $F|\partial M$ has a local maximum or a local minimum at $p$.
\end{lemma}
\begin{figure}[htbp]
\begin{center}
\includegraphics[height=.15\textheight]{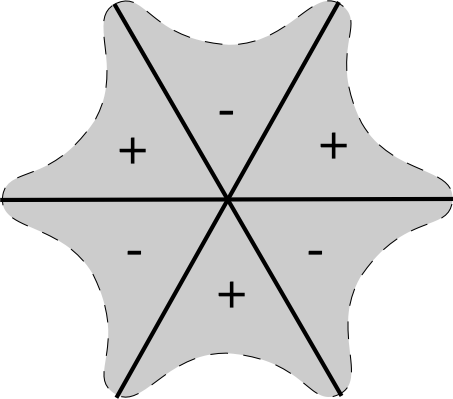} \hfill
\includegraphics[height=.25\textheight]{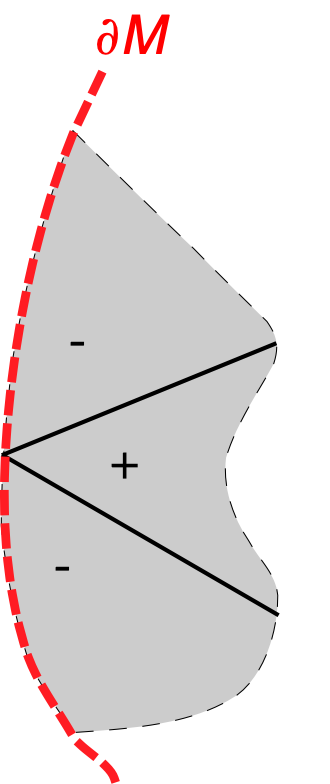} \hfill
\includegraphics[height=.25\textheight]{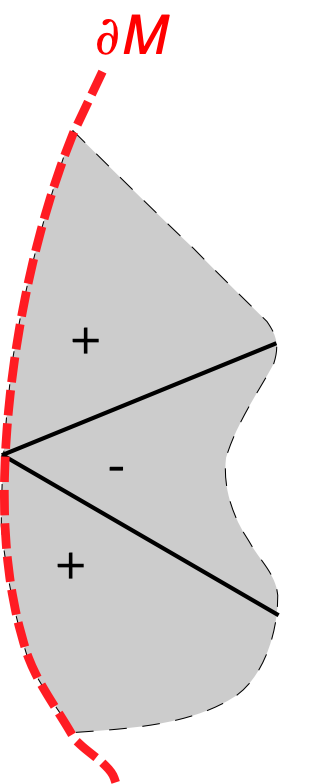}\hfill
\includegraphics[height=.25\textheight]{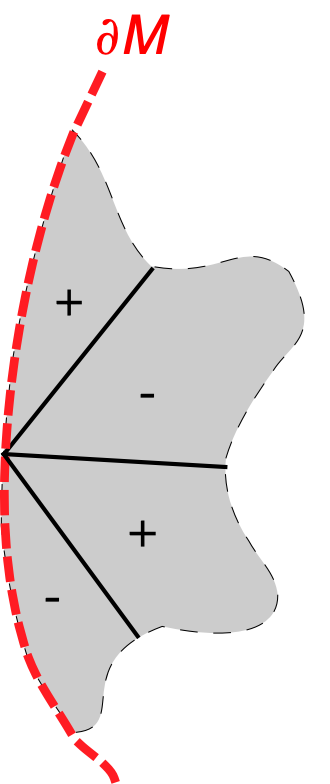}
\end{center}
\caption{From left to right: an interior critical point, a local maximum of $F|\partial M$, a local minimum of $F|\partial M$, and a point which is neither a local maximum nor a local minimum of $F|\partial M$. ``$+$'' indicates that $F>F(p)$ in that region and ``$-$'' indicates that $F<F(p)$ in that region.} \label{fig:1}
\end{figure}

\begin{definition}\label{saddle-definition}
Suppose that $F:M\to\RR$ is a Rad\'o function.
A {\bf Rad\'o critical point} (or {\bf critical point}, for short) of $F$ is an interior point $p$ such that $v(F,p)\ne 2$ or a boundary point $p$ such that $v(F,p)\ne 1$.
If $p$ is an interior point of valence $v(F,p)\ge 4$, we say that
that $p$ is a {\bf saddle of multiplicity} $w(F,p$), where
\[
   w(F,p): = \frac12v(F,p) - 1.
\]
Interior points of valence $2$ and boundary points of valence $1$ are called {\bf Rad\'o noncritical points} or {\bf Rad\'o regular points}.
\end{definition}

Two warnings about Definition~\ref{saddle-definition} are in order.
First, in case $F$ is smooth, the notion of Rad\'o critical point is not equivalent to the usual definition of critical point (i.e., a point
where $DF$ vanishes).  For example, for the Rad\'o function $F(x,y)=y^3$, every point is Rad\'o non-critical, but the points with $y=0$
are critical in the usual sense.
Second, for a general Rad\'o function, the set of Rad\'o critical points  need not be closed.  
Fortunately, under mild hypotheses, the set of Rad\'o critical points will be locally finite and therefore closed.   
See Theorem~\ref{boundary-is-rado-theorem}.  (See also Theorem~\ref{discrete-theorem}.) 

For the rest of the paper, ``critical point'', ``noncritical point", and ``regular point" will always mean ``Rad\'o critical point",
``Rad\'o noncritical point'', and ``Rad\'o regular point''.

The following theorem shows how Rad\'o functions arise naturally in minimal surface theory.

\begin{theorem}\label{minimal-surface-theorem}
Suppose $M$ is an embedded minimal surface in a smooth Riemannian $3$-manifold $N$.
Suppose $F:N\to \RR$ is a continuous function such that 
\begin{enumerate}[\upshape (1)]
\item The level sets of $F$ are smooth minimal surfaces.
\item Each level set $M[t]:=F^{-1}(t)$ is in the closure of $\{F>t\}$ and of $\{F<t\}$.\label{crossing-condition}
\end{enumerate}
Suppose also that $F$ is not constant on any connected component of $M$.
Then the restriction of $F$ to the interior of $M$ is a Rad\'o function without any interior local maxima or interior local minima.
The interior saddles of multiplicity $n$ are the points where $M$ makes contact of order $n$ with 
the level set $\{F=F(p)\}$.
\end{theorem}

(Condition~\eqref{crossing-condition} rules out examples such as $F(x_1,x_2,x_3)=|x_1|$.)

Theorem~\ref{minimal-surface-theorem}
follows from the well-known way in which two minimal surfaces in a $3$-manifold intersect each other.
See, for example,~\cite{colding-minicozzi}*{Theorem~7.3} and its proof.

\begin{remark}\label{higher-codimension-remark}
Theorem~\ref{minimal-surface-theorem} is also true for branched minimal surfaces, and for minimal surfaces
in manifolds of arbitrary dimension.  (When $\dim N>3$, the hypothesis that the level sets of $F$ are minimal is replaced
by the hypothesis that the level sets are totally geodesic.)  See \S\ref{branch-section}.
\end{remark}

\begin{remark}\label{minimal-surface-theorem-remark}
Note that Theorem~\ref{minimal-surface-theorem} only asserts that $F$ is Rad\'o on the interior of $M$.
For applications, we generally need to know that $F$ is Rad\'o on all of $M$.  Fortunately, under mild hypotheses,
a continuous function that is Rad\'o on the interior of $M$ will indeed be Rad\'o on all of $M$.
In particular, the $F|M$ in Theorem~\ref{minimal-surface-theorem} is Rad\'o on all of $M$ provided
$F|M$ is proper, $(\partial M)\cap F^{-1}(t)$ is a finite
set for each~$t$, and $d_1(M):=\dim H_1(M;\ZZ_2)<\infty$.  See Theorem~\ref{boundary-is-rado-theorem}.
Although Theorem~\ref{boundary-is-rado-theorem} appears near the end of the paper, its proof does not depend
on the intervening sections.
\end{remark}

By definition, the level sets of a Rad\'o function consist of isolated points together with curves joining them.
For a general Rad\'o function, those curves are merely continuous.
But for the functions $F|M$ in Theorem~\ref{minimal-surface-theorem}, the level sets are nicer: the curves are smooth
 (because they are transverse intersections of the smooth surface
$M$ and the smooth hypersurface $F^{-1}(t)$).  Furthermore:

\begin{theorem}\label{minimal-tame-theorem}
Suppose that $F$ and $M$ are as in Theorem~\ref{minimal-surface-theorem}. 
 Then 
\begin{enumerate}
\item\label{tame1} the set of interior non-critical points of $F|M$ is an open set.
\item\label{tame2}  At each interior non-critical point $p$, the level set $M\cap \{F=F(p)\}$ has a tangent line $\Tan(F|M,p)$, and $\Tan(F|M,p)$ depends continuously on $p$.
\end{enumerate}
Furthermore, suppose $F_n$ and $M_n$ are a sequence of such examples with $F_n$ converging uniformly to $F$ and $M_n$ converging
smoothly to $M$.   If $p$ is a non-critical point of $F|M$ and if $p_n\in M_n$ converges to $p$, then $p_n$ is non-critical for $F|M_n$ for all
sufficiently large $n$, and $\Tan(F_n|M,p_n)$ converges to $\Tan(F|M,p)$.
\end{theorem}

We omit the proof, as it follows easily from standard facts about transversality.  For example, the last sentence of 
the statement of Theorem~\ref{minimal-tame-theorem} can be reworded as follows:
If $M$ intersects $\{F=F(p)\}$ transversely at $p$, then $M_n$ intersects $\{F_n=F_n(p_n)\}$ transversely at $p_n$ for all sufficiently large $n$, and
the tangent line to 
\[
M_n\cap \{F_n=F_n(p_n) \}
\]
 at $p_n$ converges to the tangent 
line to 
\[
M\cap \{F=F(p)\}
\] at $p$.

Theorem~\ref{minimal-tame-theorem} is also true for branched minimal surfaces.  See Corollary~\ref{branch-tame-corollary}.

Rad\'o functions with properties~\eqref{tame1} and~\eqref{tame2} in Theorem~\ref{minimal-tame-theorem} are called {\bf tame}.  
(See Definition~\ref{tame-definition}.)
Tameness implies a number of other nice properties.  See~\S\ref{tame-section}.
In particular, we prove an important lower semicontinuity property (Corollary~\ref{lower-semicontinuity-corollary}). 
In the context of Theorem~\ref{minimal-tame-theorem},
it says that the number of interior critical points (i.e., saddles) of $F|M$ (counting multiplicity)
 is less than or equal to the liminf of the number 
of interior critical points of $F_n|M_n$ (counting multiplicity).

\begin{remark}\label{lipshitz-remark}
 In Theorem~\ref{minimal-tame-theorem}, $\Tan(F|M,p)$ is not merely continuous, it is actually locally Lipschitz.
(Sketch of proof: let $T(p)$ be tangent plane to $\{F=F(p)\}$ at $p$.  It is not hard
to show using the Harnack inequality that $T(\cdot)$ is locally Lipschitz.  It follows easily that $\Tan(F|M,\cdot)$ is locally Lipschitz.)
The local Lipschitz property does not play a role in this paper.
\end{remark}

\begin{remark}\label{smooth-remark} Suppose in Theorem~\ref{minimal-tame-theorem} that $F$ is smooth with nowhere vanishing gradient.
Then the function $F|M$ is particularly nice.  First, it is smooth. Second, the interior Rad\'o critical points coincide with the usual critical points
(i.e., the points where $D(F|M)$ vanishes).  Third, the multiplicity of an interior saddle point $p$ is equal to the order of vanishing of 
  $F|M - (F|M)(p)$.  
  These facts are easy to prove, but play no role in this paper.
\end{remark}

\section{Surfaces without Boundary}

\begin{lemma}\label{euler-lemma}
Let $X$ be a finite network.  Then
\[
  \chi(X) = \sum_{p\in V} \frac12(2-v(p)),
\]
where $V$ is the set of vertices and $v(p)$ is the valence of $p$.  Equivalently,
\begin{equation}\label{basic-count-2}
   \chi(X) = \sum_n \frac12(2 - n)\, |V_n| ,
\end{equation}
where $V_n$ is the set of vertices of valence $n$.
\end{lemma}

Here (and throughout the paper), if $S$ is a set, then $|S|$ denotes the number of elements of $S$.

\begin{proof}
A component without vertices is a loop, and both assertions are trivially true for such components.
Thus we can assume that every component contains one or more vertices.
Let $V$ be the set of vertices and $E$ be the set of edges.
Note that
\[
  \sum_{p\in V} v(p) = 2|E|,
\]
Thus
\[
   \chi = |V| - |E| = \sum_{p\in V} 1 - \frac12\sum_{p\in V} v(p) = \sum_{p\in V} \frac12(2-v(p)).
\]
\end{proof}

\begin{lemma}\label{graph-removal-lemma}
Let $M$ be a $2$-manifold without boundary and of finite topology, and let $X$ be a finite network in $M$
Then $M\setminus X$ has finite topology, and
\[
  \chi(M) = \chi(M\setminus X) + \chi(X). 
\] 
\end{lemma}

\begin{proof}
First remove all the vertices of $X$ from $M$ to get an open $2$-manifold $M'$ with $\chi(M')=\chi(M)-|V|$,
where $|V|$ is the set of vertices of $X$.  Note that removing a properly embedded open arc from a $2$-manifold
of finite topology increases the Euler characteristic by $1$.  Thus, removing the components of $X\setminus V$
one at a time from $M'$ produces a $2$-manifold $M''$ with 
\[
\chi(M'')=\chi(M')+|E|=\chi(M)-|V|+|E| = \chi(M)-\chi(X).
\]
\end{proof}

\begin{definition}\label{notation}
If $F:M\to \RR$ and $s\in\RR$, we let
\[
M[s] = M\cap F^{-1}(s).
\]
If $I\subset \RR$ is an interval, we let $MI:=M\cap F^{-1}(I)$.  Thus, for example,
\begin{align*}
M[s,t] &= M\cap \{s\le F\le t\}, \\
M(s,t) &= M\cap \{s < F < t\},
\end{align*}.
\end{definition}

\begin{lemma}\label{noncritical-lemma}
Suppose that $F:M\to\RR$ is a Rad\'o function and that
\begin{enumerate}[\upshape(1)]
\item There are no critical points in $M(a,b)$, and
\item $M[a',b']$ is compact for $a< a' < b' < b$.
\end{enumerate}
If $I$ is an interval in $(a,b)$ and if $t\in I$, then $M\cap F^{-1}(I)$ is homeomorphic to $M(t)\times I$.
\end{lemma}

\begin{proof}
If $M$ has no boundary, this is Corollary~\ref{annulus-corollary} in the Appendix.
The general case follows by doubling $M$.
\end{proof}

\begin{lemma}\label{level-graph-removal}
Suppose that $M$ is a compact $2$-manifold without boundary, that $F:M\to\RR$ is a Rad\'o function,
and that the set $Q$ of local maxima and local minima of $F$ is finite.
Let $\Tt\subset \RR$ be a finite set that includes $F(Q)$.
Let $X=\cup_{t\in \Tt}M[t]$.  Then
\begin{equation}\label{fundamental-inequality}
\begin{aligned}
  \chi(M) \le \chi(X) 
  &= \sum_{p\in X}\frac12(2-v(F,p)) \\
  &= |Q| + \sum_{p\in X, \, v(F,p)> 2} \frac12(2 - v(F,p)),
\end{aligned}
\end{equation}
with equality if and only if each component of $M \setminus X$ is an annulus.
In particular, if $X$ contains all the critical points of $F$, then
\begin{equation}\label{fundamental-equality}
 \chi(M) = \sum_{p\in M}\frac12(2-v(F,p)).
\end{equation}
Furthermore, if $M(a,b)$ contains no critical points, then
$M(a,b)$ is homeomorphic to $M[t] \times (a,b)$ for each $t\in (a,b)$.
\end{lemma}

\begin{proof}
By Lemma~\ref{graph-removal-lemma},
\begin{equation}\label{graph-removal-repeated}
  \chi(M) = \chi(X) + \chi(M\setminus X).
\end{equation}
Let $W$ be a component of $M\setminus X$.  Then $W$ is a component of $M(a,b)$
for two successive elements $a$, $b$ in $\Tt$.  

By Lemma~\ref{graph-removal-lemma}, $W$ is an open manifold of finite topology.  Thus it is homeomorphic to a closed
surface with finitely many points removed.  Since $W$ has no local maxima, 
\[
  \sup_WF = b.
\]
Since $W$ has no local minima, 
\[
\inf_WF=a. 
\]
 Thus $W$ is homeomorphic to  closed surface with at least two points
removed.  It follows that
\[
 \chi(W) \le 0,
\]
with equality if and only if $W$ is an annulus.  
Hence (by~\eqref{graph-removal-repeated}) the inequality~\eqref{fundamental-inequality} holds, with equality if and only if each $W$ is an annulus.

The last assertion is a special case of Lemma~\ref{noncritical-lemma}.
\end{proof}

\begin{theorem}\label{main-theorem}
Suppose that $M$ is a compact $2$-manifold without boundary
and that $F:M\to \RR$ is a Rad\'o function with a finite set $Q$ of local maxima and local minima.
Then there are only finitely many points $p$ with $v(F,p)\ne 2$, and
\begin{equation}\label{main-closed-formula}
  \chi(M) = \sum_{p\in M} \frac12(2 - v(F,p)).
\end{equation}
Equivalently,
\[
  \chi(M) = \sum_k (1 - k ) \, |V_{2k}|,
\]
where $V_n$ is the set of points $p$ such that $v(F,p)=n$.
\end{theorem}

\begin{proof}
   Note that $Q$ is the set of points of valence $0$.
Let $\Tt\subset \RR$ be a finite set that includes $F(Q)$. 
 Let $X=\cup_{t\in \Tt}M[t]$.
By Lemma~\ref{level-graph-removal},
\begin{align*}
\chi(M)
&\le |Q| + \sum_{p\in X, \, v(F,p)> 2} \frac12(2 - v(F,p))  \\
&\le |Q| - |\{p\in X: v(F,p)>2\}|  \\
&= 2|Q| - |\{p\in X: v(F,p)\ne 2\}| \\
&= 2|Q| - |C|
\end{align*}
where $C$ is the set of critical points of $F$ in $X$.
Thus $X$ has at most 
\[
     2|Q| - \chi(M)
\]
critical points.  Since this bound holds for every such $\Tt$, we see that $M$ has at most $2|Q|-\chi(M)$ critical points.  
Equation~\eqref{main-closed-formula} now follows from~~\eqref{fundamental-equality} in
Lemma~\ref{level-graph-removal} by letting 
\[
  \Tt = \{ F(p): v(F,p)\ne 2\}.
\]
\end{proof}

\begin{corollary}\label{maxwell-corollary}
The number of saddle points, counting multiplicity, is equal to the number of local maxima and local minima minus the Euler 
characteristic:
\[
  \sum_{w(F,p)>0} w(F,p) = |Q| - \chi(M).
\]
\end{corollary}

\begin{proof}
Each point of valence $0$ (i.e., each point of $Q$) contributes $1$ to the sum in~\eqref{main-closed-formula},
each point of valence $2$ contributes $0$, and there are no points of odd valence.  
Thus~\eqref{main-closed-formula} becomes
\begin{align*}
  \chi(M) 
  &= |Q| + \sum_{v(F,p)\ge 4} \frac12(2- v(F,p))
  \\
  &= |Q| - \sum_{w(F,p)>0} w(F,p).
\end{align*}
\end{proof}

A version of Corollary~\ref{maxwell-corollary} in the case of the $2$-sphere
 occurs in an 1870 paper~\cite{maxwell} by the physicist Maxwell.
In particular, Maxwell does allow saddles with multiplicity.

\section{Surfaces with Boundary}

\newcommand{\Vint}{V^\textnormal{int}}
\newcommand{\Vb}{V^\partial}

\begin{theorem}\label{first-version-boundary-theorem}
Suppose that $M$ is a compact $2$-manifold with boundary, that $F:M\to\RR$ 
is a Rad\'o function, and that there are only finitely many points of valence $0$.
Then
\[
   \chi(M)  =  \sum_{p\in M\setminus \partial M} \frac12(2-v(F,p))  + \sum_{p\in \partial M} \frac12 (1 - v(F,p)).
\]
Equivalently,
\[
    \chi(M) = \sum_{k=0}^\infty\left( (1-k)|\Vint_{2k}| +  \frac12(1-k)|\Vb_k| \right),
\]
where $\Vint_n$ is the set of interior points $p\in M\setminus \partial M$ with $v(F,p)=n$,
and $\Vb_n$ is the set of boundary points $p\in \partial M$ with $v(F,p)=n$. 
\end{theorem}

\begin{proof}
Let $\tilde M$ be the closed manifold obtained by doubling $M$. That is, we take two copies of $M$ and attach them
along their boundary. 
 Let $\tilde F$ be the obvious extension of $F$ to $\tilde M$.  
 Since $F:M\to\RR$ is Rad\'o, it follows easily that $\tilde F:\tilde M\to \RR$ is also Rad\'o.
 Let $\tilde V_n$ be the set of points $p\in\tilde M$
with $v(\tilde F,p)=n$.  Then 
\[
   |\tilde V_{2k}|  = 2 |\Vint_{2k}|  + |\Vb_k|.
\] 
Applying Theorem~\ref{main-theorem} to $\tilde M$ gives
\begin{align*}
\chi(M)
&=
\frac12\chi(\tilde M) \\
&=
\frac12\sum_k (1-k)\,|\tilde V_{2k}|   \\
&=
\frac12\sum_k (1-k)(2 |\Vint_{2k}| + |\Vb_k|) \\
&=
\sum_k \left( (1-k)|\Vint_{2k}| + \frac12(1-k)|\Vb_k| \right).
\end{align*}
\end{proof}

The statement of Theorem~\ref{first-version-boundary-theorem} is
fairly simple.  However, the theorem can be rewritten in a way
that makes it easier to use.

\begin{theorem}\label{main-boundary-theorem}
Suppose that $F:M\to \RR$ is a Rad\'o function on a compact $2$-manifold with boundary.
Let $Q$ be 
\begin{enumerate}[\upshape(i)]
\item  the set of interior local maxima and interior local minima of $F$, together with
\item the set of local {\bf minima} of $F|\partial M$.
\end{enumerate}
Suppose that $Q$ is a finite set.
Then
\begin{equation}\label{general-formula}
\sum_{n\ge 2} (n-1) \,|\Vint_{2n}| 
+
\sum_{n\ge 1} n (  |\Vb_{2n}|  + |\Vb_{2n+1}|)  
=
|Q| - \chi(M).
\end{equation}
\end{theorem}

This way of rewriting Theorem~\ref{first-version-boundary-theorem} is very useful for the following reason.
Think of $M$ and $F|\partial M$ as given, and the function $F$ as unknown.
In many situations (such as for minimal surfaces in Theorem~\ref{minimal-surface-theorem}),
 we know that there are no interior local maxima or minima.  In that case, $Q$ is the set of local minima of $F|\partial M$, 
which we regard as known.  Thus the right hand side is known, and
the terms on the left are all positive.

Recall that interior points of valence $v>2$ are called interior saddle points of multiplicity $w(F,p):=(v/2)-1$.

\begin{definition}\label{boundary-saddle-definition}
A {\bf boundary saddle point} of a Rad\'o function $F$ is a boundary point of valence $>1$. 
The {\bf multiplicity} of a boundary saddle point is
\[
w(F,p) 
:= 
\begin{cases}
v/2  &\text{if $v$ is even}, \\
(v-1)/2  &\text{if  $v$ is odd}.
\end{cases}
\]
\end{definition}

Using this definition, Theorem~\ref{main-boundary-theorem} can be restated as follows:

\begin{theorem}\label{restated-theorem}
Under the hypotheses of Theorem~\ref{main-boundary-theorem}, 
\begin{equation}\label{general-formula-weights}
  \sum_{w>0} w(F,p) = |Q| - \chi(M).
\end{equation}
\end{theorem}

Note that the left hand side is the total number of saddles, interior and boundary, counting multiplicity.

\begin{proof}[Proof of Theorem~\ref{main-boundary-theorem}]
Write
\[
  Q = \Qint + \Qb,
\]
where $\Qint:=Q\setminus \partial M$ is the set of interior local maxima and interior local minima of $F$, and 
where $\Qb:=Q\cap \partial M$ is the set of local minima of $F|\partial M$.

Note that the points of valence $0$ are the points of $Q$ together with the local maxima of $F|\partial M$. 
Since the number of local maxima of $F|\partial M$ is equal to the number $|\Qb|$ of local minima of $F|\partial M$,
we see that there are only finitely many points of valence $0$.

Recall from Theorem~\ref{first-version-boundary-theorem} that
\begin{equation}\label{eeny}
\chi(M) = \sum_k (1-k)|\Vint_{2k}|  + \frac12\sum_k (1-k) |\Vb_k|.
\end{equation}
Now
\begin{equation}\label{meeny}
\begin{aligned}
\sum_k (1-k)|\Vint_{2k}|
&=
|\Qint| + 0 - \sum_{k\ge 2} (k-1)\,|\Vint_{2k}|.
\end{aligned}
\end{equation}
Also
\begin{equation}\label{miny}
\begin{aligned}
\sum_k (1-k)|\Vb_k|
&=
\sum_n ((1-2n)|V_{2n}| + (-2n)|V_{2n+1}| ) \\
&=
\sum_n|\Vb_{2n}|  - \sum_n (2n)(|V_{2n}| + |V_{2n+1}| ). 
\end{aligned}
\end{equation}

By Lemma~\ref{parity-lemma}, 
\[
  \sum_n |\Vb_{2n}|
\]
is the number of local minima and local maxima of $F|\partial M$.
The number of local maxima of $F|\partial M$ is equal to the number of local minima of $F|\partial M$ (namely $|\Qb|$), so
\begin{equation*}
  \sum_n|\Vb_{2n}| = 2|\Qb|.
\end{equation*}
Thus we can rewrite~\eqref{miny} as
\begin{equation}\label{mo}
\frac12 \sum_k (1-k)|\Vb_k|
=
 |\Qb|  -  \sum_n n \, (|V_{2n}| + |V_{2n+1}| ).  
\end{equation}

Combining~\eqref{eeny}, \eqref{meeny}, and~\eqref{mo} gives~\eqref{general-formula}.
\end{proof}

\section{A Remark about Inequalities}

Various theorems in this paper, such as Theorem~\ref{restated-theorem},
give formulas for the total number of saddles, interior and boundary, in some region, counting multiplicity.
For many applications, simpler inequalities suffice.  

The following proposition describes how the exact formulas imply the simpler inequalities.

\begin{proposition}\label{inequality-proposition}
Suppose that $F:M\to\RR$ is a Rad\'o function and that $K$ is a region in $M$.  
Suppose also that
\[
   \sum_{K \cap \{w>0\}} w(F,p) = \Ww.
\]
Then
\begin{equation}\label{basic-inequality}
\sum_{(K\setminus \partial M)\cap\{w>0\}} w(F,p)
\le  \Ww -  |A|,
\end{equation}
where $A$ is the set of points $p$ in $K\cap \partial M$ such that $p$ is a local minimum or local maximum of $F|\partial M$
but is not a local minimum or local maximum of $F$.

Furthermore, equality holds if and only if  $K\cap \partial M$ contains no point $p$ with valence $v(F,p)>2$.
In particular, if $F$ is $C^2$, if $F|M$ is a Morse function, and if $DF$ does not vanish 
   at any point of $\partial M$,
then equality holds in~\eqref{basic-inequality}.
\end{proposition}

\begin{proof}
Note that
\[
  \sum_{K\cap \{w>0\}} w(F,p) = \sum_{(K\setminus\partial M)\cap \{w>0\}} w(F,p) + \sum_{(K\cap\partial M)\cap \{w>0\}} w(F,p),
\]
and that
\begin{align*}
\sum_{(K\cap\partial M)\cap \{w>0\}} w(F,p)  
&=
\sum_{(K\cap \partial M)\cap \{v\ge 2\}} w(F,p)  \\
&\ge
\sum_{(K\cap \partial M) \cap \{v\ge 2, \,\text{$v$ even} \}} w(F,p)  \\
&=
\sum_{(K\cap \partial M)\cap \{v\ge 2,\, \text{$v$ even} \}} \frac12 v(F,p) \\
&\ge
\sum_{(K\cap \partial M)\cap \{v\ge 2, \,\text{$v$ even} \}} 1 \\
&= |A|
\end{align*}
with equality if and only $K\cap \partial M$ has no points of valence $>2$.

This proves the proposition, except for the assertion about the case when $F$ is $C^2$.
Note that if $F$ is $C^2$ and if $F|\partial M$ is a Morse function, then at each point of $\partial M$ that is 
not a critical point of $F|\partial M$, the valence is $1$, and at each critical point of $F|\partial M$, the valence is either $0$ or $2$.
\end{proof}

Thus, for example, from Theorem~\ref{restated-theorem}, we get the following inequality:

\begin{theorem}\label{restated-inequality-theorem}
Suppose that $F:M\to \RR$ is a Rad\'o function on a compact $2$-manifold with boundary.
Let $Q$ be the set of interior local maxima and interior local minima of $F$, together with
the local minima of $F|\partial M$.
Suppose that $Q$ is a finite set.
Then
\begin{equation}\label{general-inequality}
\sum_{(M\setminus \partial M)\cap\{w>0\}} w(F,p)
\le 
|Q| - \chi(M) - |A|,
\end{equation}
where $A$ is the set of local maxima and local minima of $F|\partial M$ that are not local maxima or local minima of $F$.
Equality holds if and only if there are no boundary points $p$ of valence $v(F,p)>2$.

In particular, if $F$ is $C^2$, if $F|\partial M$ is a Morse function, and if $DF$ does not vanish at 
any point of $F|\partial M$, then equality holds.
\end{theorem}

\section{Portions of Surfaces with Boundary}

\begin{theorem}\label{regular-interval-theorem}
Suppose that $F:M\to \RR$ is a Rad\'o function, that
$a < b$ are regular values of $F$, 
and that $M[a,b]$ is compact.
Then
\begin{equation}\label{regular-interval-formula}
  \sum_{M(a,b)\cap \{w>0\}} w(F,p) = |Q(a,b)| + \frac12\beta(a)  -  \chi(M(a,b)),
\end{equation}
 provided $|Q(a,b)|$ is finite,
where
\begin{enumerate}
\item $Q$ is the set consisting of the interior local maxima and the interior local minima of $F$,
together with the local minima of $F|\partial M$, 
\item $Q(a,b)= Q\cap M(a,b)$, and
\item $\beta(a)$ is the number of points in $(\partial M)\cap \{F=a\}$.
\end{enumerate}
\end{theorem}

\begin{proof}
Let $\tilde M$ be obtained from $M[a,b]$  by identifying each
connected component of $M[a]$ to a a point and each connected component of $M[b]$ to a point.
Let $\tilde F$ be the function on $\tilde F$ corresponding to $F$ on $M[a,b]$.

Thus each closed curve component of $M[a]$ becomes an interior point of $\tilde M$,
and each non-closed curve component of $M[a]$ becomes a single boundary point of $\tilde M$.
In both cases, the point is a global minimum of $\tilde F$.

Likewise, each closed curve component of $M[b]$ becomes an interior point of $\tilde M$,
and each non-closed curve component of $M[b]$ becomes a single boundary point of $\tilde M$.
In both cases, the point is a global maximum of $\tilde F$.

Let $\tilde Q$ be the set of all interior local maxima and
interior local minima of $\tilde F$, together will all local minimal of $\tilde F|\partial \tilde M$.

Let $n$ be the number of non-closed-curve components of $M[a]$, and let $c$ be the number of closed curved components
of $M[a]\cup M[b]$.

Note that
\begin{align*}
\chi(\tilde M) &= \chi(M(a,b)) + c,  \\
|\tilde Q|  &= |Q(a,b)| + n + c, \,\text{and} \\
n &= \frac12 \beta(a).
\end{align*}
Thus
\[
  |\tilde Q| - \chi(\tilde M) =  |Q(a,b)| - \chi(M(a,b)) + \frac12 \beta(a).
\]
Consequently, by Theorem~\ref{restated-theorem},    
\begin{equation}\label{hansel}
\begin{aligned}
  \sum_{\tilde M\cap \{w>0\}} w(\tilde F, p)  
  &=  |\tilde Q| - \chi(\tilde M) \\
  &= |Q(a,b)| + \frac12\beta(a) - \chi(M(a,b)).
\end{aligned}
\end{equation}
The points in $\tilde M$ with $\tilde F=a$ or $\tilde F=b$ all have valence $0$, so
\begin{equation}\label{gretel}
  \sum_{\tilde M\cap \{w>0\}}w(\tilde F,p) = \sum_{M(a,b)\cap \{w>0\}}w(F,p).
\end{equation}
Combining~\eqref{hansel} and~\eqref{gretel} gives~\eqref{regular-interval-formula}.  
\end{proof}

We now relax the requirement in Theorem~\ref{regular-interval-theorem} that $a$ and $b$ are
finite, non-critical values of $F$.  We begin with a Lemma.

\begin{lemma}\label{monomorphism-lemma}
Suppose that $M$ is a surface and that $F:M\to \RR$ is a continuous function.
Suppose also that $F$ has no interior local minima with $F<a,$ and no interior local maxima with $F>b$.
 Then the inclusion of $M[a,b]$ into $M$ induces a monomorphism on $H_1(\--; \ZZ_2)$.
 
Likewise, if $F$ has no interior local minima with $F\le a,$ and no interior local maxima with $F\ge b$,
then inclusion of $M(a,b)$ into $M$ induces a monomorphism on $H_1(\--;\ZZ_2)$.
\end{lemma}

\begin{proof}
We prove the first statement.
Let $C$ be a $1$-cycle in $M[a,b]$ that is homologically trivial in $H_1(M; \ZZ_2)$.
Then $C$ bounds a region $K$ in $M$.  
Let $p$ be a point where $F|K$ attains its maximum.  If $p\in C$, then $F(p)\le b$ since $C\subset M[a,b]$.
If $p\in K\setminus C$, then $p$ is an interior local maximum of $F$ and hence $F(p)\le b$.  Either way, 
$\max_KF\le b$.  Likewise, $\min_KF\ge a$.  
Hence $K$ lies in $M[a,b]$, so $C$ is homologically
trivial in $H_1(M[a,b]; \ZZ_2)$.
\end{proof}

 For the next two theorems, we make the following hypotheses:
\begin{enumerate}[\upshape (h1)]
\item\label{h1} $F:M\to\RR$ is a Rad\'o function and $-\infty\le a < b\le \infty$.
\item\label{h2} $d_1(M):=\dim H_1(M;\ZZ_2)$ is finite.
\item\label{h3} 
  The set $Q$  is finite, where $Q$ consists of the interior local minima and maxima of $F$ together with
the local minima of $F|\partial M$.
\end{enumerate}

\begin{theorem}\label{semiregular-interval-theorem}
Under the hypotheses~(h\ref{h1})--(h\ref{h3}), if $M[a,t]$ is compact for all $a\le t<b$ and if $a$ is a regular value of $F$, then
\[
  \sum_{M(a,b)\cap \{w>0\}} w(F,p) =  |Q(a,b)| + \frac12\beta(a)  -  \chi(M(a,b)),
\]
where  $\beta(t)$ is the number of points in $(\partial M)\cap \{F=t\}$.
\end{theorem}

\begin{proof}
Let $Z$ be the set of interior local maxima and minima of $F$ and let $Z^*(a,b)=Z\cap \{F\notin (a,b)\}$.
Note that $Z\subset Q$ and that
\[
   M(s,t) \subset M\setminus Z^*(s,t)
\]
 induces a monomorphism of first homology (see Lemma~\ref{monomorphism-lemma}), so
\begin{align*}
 d_1(M(s,t))
 &\le d_1(M\setminus Z^*(s,t)) \\
 &= d_1(M)+|Z^*(s,t)|, \\
 &\le d_1(M) + |Q|.
\end{align*}
Therefore,
\[
 -\chi(M(s,t)) \le d_1(M(s,t)) \le d_1(M)+|Q|.
\]
If $t\in (a,b)$ is a regular value of $F$, then by Theorem~\ref{regular-interval-theorem},
\begin{align*}
\sum_{p\in M(a,t), \, w>0} w(F,p)
&=
|Q(a,t)| + \frac12 \beta(a) - \chi(M(a,t)) \\
&\le
 \frac12\beta(a) +  d_1(M) +|Q|.
\end{align*}
Note that this final expression is indepent of $t$.  
By elementary topology (see Corollary~\ref{n-ad-corollary}), there are at most countably many
critical points and hence at most countably many critical values.
Thus (letting $t\to b$ among regular values $t$),
\[
 \sum_{p\in M(a,b), \, w>0} w(F,p) \le \frac12\beta(a) + d_1(M) + |Q| < \infty.
\]
Hence the set $S:=\{p\in M(a,b): w(F,p)>0\}$ is finite.
The set $Q(a,b)$ is also finite, so we can choose a regular value $b'$ of $F$ in $(a,b)$
such that
\[
  S \cup Q(a,b) \subset M(a,b').
\]
Now $S\cup Q(a,b)$ contains all the critical points of $F|M(a,b)$. Thus there are no critical points in $M[b',b)$.
Consequently, $M(a,b)$ is homotopy equivalent to $M(a,b')$ (see Lemma~\ref{noncritical-lemma}), so
\[
   \chi(M(a,b')) = \chi(M(a,b)).
\]

By Theorem~\ref{regular-interval-theorem},
\begin{align*}
\sum_{p\in M(a,b),\, w>0} w(F,p) 
&=
\sum_{p\in M(a,b'),\, w>0} w(F,p) \\
&=
|Q(a,b')| +\frac12\beta(a) - \chi(M(a,b')) \\
&=
|Q(a,b)| + \frac12\beta(a) - \chi(M(a,b)).
\end{align*}
\end{proof}

\begin{theorem}\label{general-interval-theorem}
Under the hypotheses~(h\ref{h1})--(h\ref{h3}),
if $M[s,t]$ is compact for all $a<s<t<b$, and if the limit
\[
  \beta(a+)=\lim_{t\to a, \, t>a} \beta(t)
\]
exists and is finite, then
\[
  \sum_{M(a,b)\cap \{w>0\}} w(F,p) =  |Q(a,b)| + \frac12\beta(a+)  -  \chi(M(a,b)).
\]
\end{theorem}

\begin{proof}[Proof of Theorem~\ref{general-interval-theorem}]
Since $\beta(t)$ is integer-valued, there is an $a'\in (a,b)$ such that
\[
  \beta(t)= \beta(a+)   \quad \text{for $a< t \le a'$}.
\]

Now let $s\in (a,a']$ be a regular value of $F$.  Then (by Theorem~\ref{semiregular-interval-theorem})
\begin{align*}
\sum_{M(s,b)\cap \{w>0\}} w(F,p) 
&=
|Q(s,b)| + \frac12 \beta(t) - \chi(M) \\
&\le
|Q(a,b)| + \frac12\beta(a+) - d_1(M).
\end{align*}
Since this last expression is finite and independent of $s$, letting $s\to a$ 
(among regular values $s$) gives
\[
  \sum_{p\in M(a,b),\, w>0} w(F,p) <\infty.
\]
Thus the set $S$ of points in $M(a,b)$ where $w>0$ is finite.
The set $Q(a,b)$ is also finite.
By replacing $a'$ by smaller noncritical value in $(a,b)$, we can assume that
\[
  S\cup Q(a,b) \subset M(a',b).
\]
Now $S\cup Q(a,b)$ contains all the critical points of $F$ in $M(a,b)$.  
Thus there are no critical points in $M(a,a']$, so $M(a,b)$ is homotopy equivalent to $M(a',b)$ (by Lemma~\ref{noncritical-lemma}).
Consequently,
\[
  \chi(M(a,b)) = \chi(M(a',b)).
\]
Thus
\begin{align*}
\sum_{p\in M(a,b),\, w>0} w(F,p)
&=
\sum_{p\in M(a',b),\,w>0} w(F,p) \\
&=
|Q(a',b)| + \frac12\beta(a') - \chi(M(a',b)) \\
&=
|Q(a,b)| + \frac12 \beta(a+) - \chi(M(a,b)).
\end{align*}
\end{proof}

\begin{corollary}\label{general-interval-corollary}
In Theorem~\ref{general-interval-theorem}, if $F:M\to(a,b)$ is proper, then
\[
  \sum_{M\cap \{w>0\}} w(F,p) =  |Q| + \frac12\beta(a+)  -  \chi(M).
\]
\end{corollary}

\newcommand{\Reg}{\operatorname{Reg}}
\newcommand{\Hopf}{\operatorname{Hopf}}

\section{Tame Rad\'o Functions}\label{tame-section}

\begin{definition}\label{tame-definition}
Suppose that $F : M \to \RR$ is a Rad\'o function. 
If $p$ is an interior regular point of $F$, we let $\Tan(F, p)$ be the tangent line to $\{F = F (p)\}$ at $p$, if the tangent
line exists.
We say that $F$ is {\bf tame} provided: 
\begin{enumerate}[\upshape(1)]
\item The set of interior regular points (i.e., the set of interior points of valence $2$) is open, 
\item $\Tan(F,p)$ exists at each interior regular point, and
 $\Tan(F, \cdot)$ is a continuous function on the
 set of interior regular points.
 \end{enumerate}
\end{definition}

By Theorem~\ref{minimal-tame-theorem}, the Rad\'o functions that arise in minimal surface theory are tame.

\begin{theorem}\label{discrete-theorem}
Suppose that $F:M\to\RR$ is a tame Rad\'o function such that the set $\Qint$ of interior local minima and interior local maxima 
is closed and discrete.  Then the set of interior critical points is closed and discrete.  
In other words, each interior point $p$ has a neighborhood $U$ such that $U\setminus\{p\}$ contains no critical points.
\end{theorem}

\begin{proof}
Let $p$ be an interior critical point.  Thus $v(F,p)$ is an even number $\ne 2$.

{\bf Case 1}: $v(F,p)=0$.  Then $p$ is a local maximum or local minimum.
We may assume that it is a local minimum.
Let $K$ be a compact set such that $p$ is the interior of $K$, such that $\displaystyle \min_{\partial K}F>F(p)$,
and such that $K\setminus\{p\}$ contains no local minima or local maxima of $F$.
Choose $t$ with
\[
   F(p) < t < \min_{\partial K} F.
\]
Let $D:= K\cap \{F\le t\}$.
Then $D$ is a disk, so if we identify $\partial D$ to a point, we get a topological sphere $\Sigma$
on which $F$ is a well-defined Rad\'o function.  Note that $F|\Sigma$ has exactly one local maximum
and one local minimum.  Thus by Corollary~\ref{maxwell-corollary}, $F|\Sigma$ has no saddle points.  Therefore $F$ has
 no critical points on $D\setminus \{p\}$. 
This completes the proof in Case 1.

{\bf Case 2}: $v(F,p)=2k>0$.

We may assume that $F(p)=0$. 
Since the result is local, we can assume that $M$ is a disk, that $M$ has no interior local maxima or local minima,
and that 
\[
  \{F=0\}) \setminus \{p\}
\]
consists of $2k$ disjoint, embedded $C_1$ curves, each joining $p$ to a point in $\partial M$ (see Fig. 2.)
\begin{figure}[htbp]
\begin{center}
\includegraphics[height=.25\textheight]{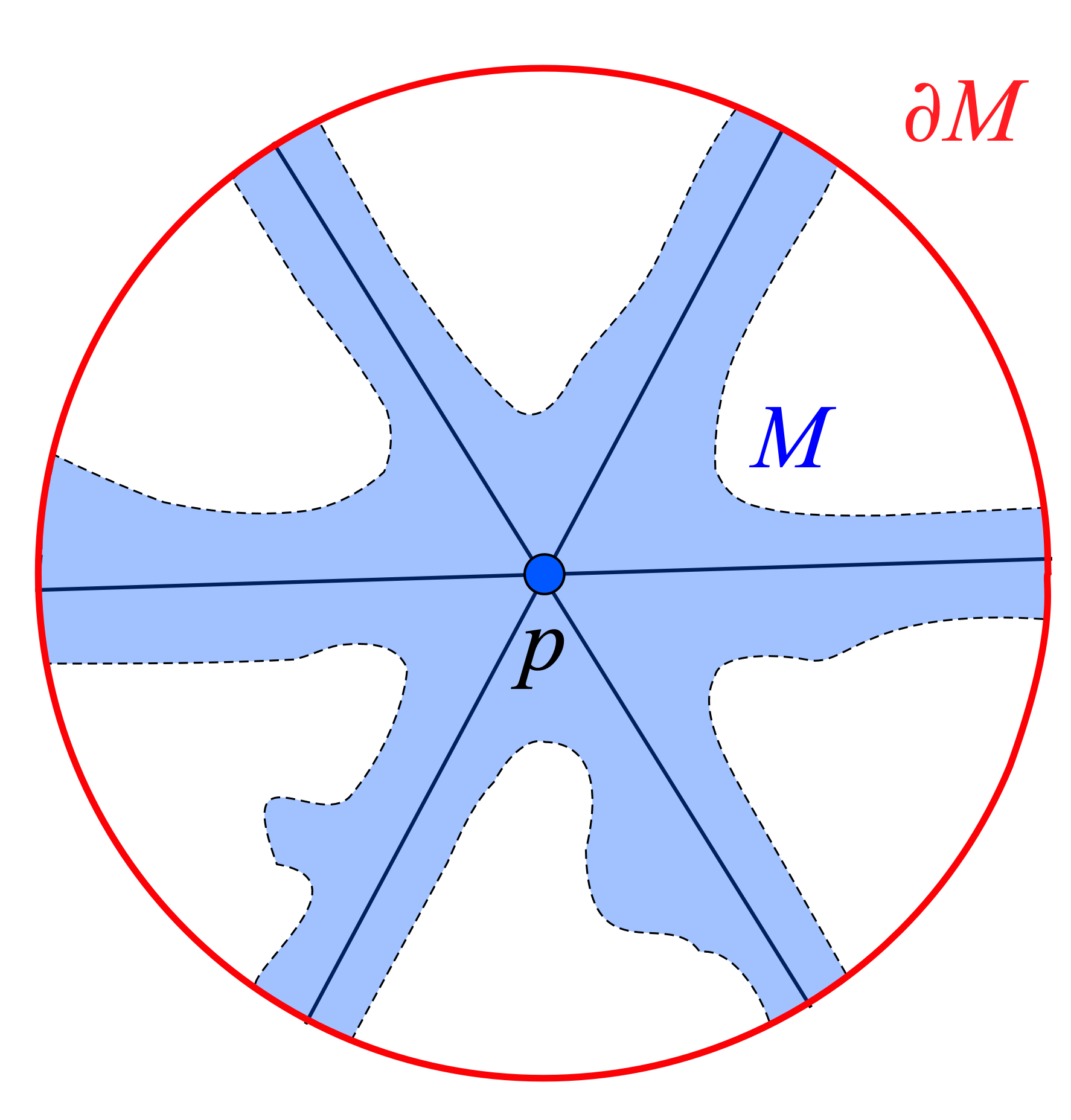}
\end{center}
\caption{Since  Case 2 is local,  we can assume that $M$ is a disk, that $M$ has no interior local maxima or local minima,
and that $  \{F=0\}) \setminus \{p\}$ consists of $2k$ disjoint, embedded $C_1$ curves, each joining $p$ to a point in $\partial M$ } \label{fig:2}
\end{figure}

By applying a homeomorphism from $M$ into $\RR^2$ that is $C^1$ in $M\setminus \{p\}$, we can assume that
\begin{gather*}
   M = \{q\in \RR^2: |q| < 2\}, \\
   p=0,
\end{gather*} 
and that 
\[
   F(r\cos \theta, r\sin\theta)=0 \iff \text{$\theta$ is an integral multiple of $\pi/k$}.
\]

Let $\Delta$ be the region given in polar coordinates by $0<r\le 1$ and $0< \theta < \pi/k$.
Note that $\Delta$ is one of the components of 
\[
  \BB(0,1) \setminus \{F=0\}.
\]
We may assume that $F>0$ on $\Delta$.
By tameness, the unit circle 
is tranverse to $\Tan(F,\cdot)$ near the points $(\cos(j\pi/k), \sin(j\pi/k))$.

In particular, we can choose $\eps>0$ so that the unit circle is transverse
to $\Tan(F,\cdot)$ at $(\cos\theta, \sin \theta)$ for $\theta \in [0,\eps]$ and for $\theta \in [(\pi/k)-\eps, (\pi/k)]$.
Thus $F(\cos\theta,\sin\theta)$ is strictly increasing for $0\le \theta \le \eps$
and is strictly decreasing for $(\pi/k)-\eps\le \theta \le \pi/k$.

Let
\[
   \eta = \min \{ F(\cos\theta, \sin\theta): \eps \le \theta \le (\pi/k)  - \eps\}.
\]
Hence for $0< t < \eta$, there are exactly
two points $q$ and $q'$ in $\partial D$ at which $F=t$.

\begin{claim}\label{easiest-Rad\'o-claim}
 $\Gamma(t):=\Delta \cap \{F=t\}$ consists of a curve of non-critical  points joining $q$ to $q'$.
\end{claim}

\begin{proof}[Proof of Claim~\ref{easiest-Rad\'o-claim}]
This is the well-known argument of Rad\'o.
  (See Theorem~\ref{slice-theorem} for a very general form of Claim~\ref{easiest-Rad\'o-claim}.)
We know that $\Gamma(t)$ is a network.  Since it is contained in a compact subset of the interior of $M$,
it is a finite network.  It cannot contain a closed curve, since if it did, that closed curve would bound a disk in $M$
(since $M$ is simply connected), and $F$ on that disk would its mininum and/or its maximum at an interior point,
which is impossible since we are assuming that there are no interior local maxima or minima.
  Thus $\Gamma(t)$
is a tree.  Since the tree has at most two points of valence $\le 1$ (namely $q$ and $q'$), it is a curve joining $q$ to $q'$.
\end{proof}

We have shown: for small enough $\eta$, there are no critical points in $\overline{D}$ with
$F<\eta$.  The same argument in the other components of $M\setminus \{F=0\}$ shows that
there is an $\eta'>0$ such there are no critical points in $\overline{D}$ with $0<|F|<\eta'$.
Also, all the points on $\{F=0\}\setminus \{p\}$ are regular.  
\end{proof}

\begin{remark}\label{untamed-example}
Without the hypothesis of tameness, Theorem~\ref{discrete-theorem} is false.
Consider the harmonic function $h(x,y)=y - (\cosh x)(\sin y)$, and let 
\[
  F(x,y)
  =
  \begin{cases}
  h(x,y^{-1})^{-1}  &\text{if $y\ne 0$, and}\\
  0 &\text{if $y=0$.}
  \end{cases}
\]
Note that the class of Rad\'o functions is closed under composition with homeomorphisms of the domain
and of $\RR$.  Since $h$ is harmonic, it is Rad\'o, and thus $F$ is Rad\'o on $\RR^2\setminus\{y=0\}$.  We leave
it to the reader to check that $F$ is Rad\'o on all of $\RR^2$.
Thus $F$ is a proper, Rad\'o function on $[-1,1]\times\RR$.  
The interior critical points of $F$ are the points $(0, (2\pi n)^{-1})$ where $n$ is an integer.
Thus the non-critical point $(0,0)$ is a limit of critical points.
\end{remark}

\begin{theorem}\label{hopf-theorem}
Suppose that $F:M\to\RR$ is a tame Rad\'o function and that $p$ is an interior point.
Then 
\[
   \Hopf(F,p) = \left( 1 - \frac{v(F,p)}2 \right) = - w(F,p),
\]
where $\Hopf(F,p)$ is the Hopf index of $\Tan(F,\cdot)$ at $p$.
\end{theorem}

\begin{proof}
We use the notation in the proof of Theorem~\ref{discrete-theorem}.
In that proof, we can modify $F$ near the unit circle $\partial \BB$ so that the level sets of $F$ cross the circle orthogonally when
$|F|\le \delta$ for some small $\delta>0$.
By the proof of Theorem~\ref{discrete-theorem}, we can choose $\delta$ small enough so that 
\[
    N:=\{q\in \BB(0,1): |F(q)|\le \delta \}
\]
contains no critical points other than $0$.  Note that $N$ is bounded by $4k$ arcs, $2k$ in the circle $\partial D$
and the other $2k$ in the interior of the unit disk.  Now we can invert $N$ in the unit circle to get a $2$-manifold with boundary $\tilde N$
in $\RR^2\cup \{\infty\}$.   We extend $F$ and $\Tan(F,\cdot)$ to $\tilde N$ by inversion.
Note that $\Tan(F,\cdot)$ is tangent to $\partial \tilde N$, so by the Poincare-Hopf Index Theorem,
\[
  \chi(\tilde N) = \Hopf(F,0) + \Hopf(F,\infty) = 2 \, \Hopf(F,0).
\]
Now $\tilde N$ is $\Ss^2$ with $2k$ disjoint open disks removed,
so $\chi(\tilde N)=2-2k$.  Thus
\[
  2-2k = 2\, \Hopf(F,0).
\]
\end{proof}

\newcommand{\domain}{\operatorname{domain}}

Suppose that $M$ is a locally compact space, that $U_n$ and $U$ are open subsets of $M$, and that $\phi_n:U_n\to V$ and
 $\phi:U\to V$ are continuous maps 
to a metrizable space $V$.  We say that $\phi_n$ converges locally uniformly to $\phi$ provided the following holds:
if $p\in \domain(\phi)$ and if $p_n\to p$, then $p_n\in \domain(\phi_n)$ for all sufficiently large $n$ and 
$\phi_n(p_n)\to \phi(p)$.  It follows that if $\phi_n$ converges locally uniformly to $\phi$ 
and if $K$ is a compact subset of $\domain(\phi)$, then $K\subset \domain(\phi_n)$ for all sufficiently large $n$, 
and $\phi_n|K$ converges uniformly to $\phi|K$.

\begin{theorem}\label{continuity-theorem}
Suppose $F_n: M_n\to \RR$ and $F: M\to \RR$ are tame Rad\'o functions, where $M_n$ is an exhaustion of $M$,
such that $\Tan(F_n,\cdot)$ converges locally
uniformly to $\Tan(F,\cdot)$.   Suppose that $K$ is a compact region of $M\setminus \partial M$ such
that $\partial K$ is contained in the regular set of $F$.  Then for all sufficiently large $n$,
\begin{equation}\label{hopf-counts}
   \sum_{p\in K} \left(1 - \frac{v(F_n,p)}2 \right)  = \sum_{p\in K} \left(1 - \frac{v(F_n,p)}2 \right).
\end{equation}
Equivalently,
\[
    \sum_{p\in K} w(F_n,p)  = \sum_{p\in K} w(F,p).
\]
\end{theorem}

\begin{proof}
First, consider the case when $K$ is topologically a disk.  Then we can choose local coordinates
so that $K$ is a disk in $\RR^2$.  By the Poincare-Hopf Theorem, the right-hand side of~\eqref{hopf-counts}
is equal to the degree of the map
\newcommand{\RP}{\mathbf{RP}}
\[
   q\in \partial K \mapsto \Tan(F,q) \in \RP^1,
\]
and the left side is equal to the degree of 
\[
   q\in \partial K \mapsto \Tan(F_n,q) \in \RP^1,
\]
By hypothesis, the two maps are homotopic for all sufficiently large $n$, and thus the two degrees are equal.

In the general case, let $D$ be a finite union of disjoint closed disks in the interior of $K$
such that the interior $U$ of $D$ contains all the critical points of $F$ in $K$.
Now $K\setminus U$ is contained in $\Reg(F)$, so it is contained in $\Reg(F_n)$ for all sufficiently large $n$.
Thus by the simply-connected case,
\begin{align*}
\sum_{p\in K} \left(1 - \frac{v(F_n,p)}2 \right)
&=
\sum_{p\in D} \left(1 - \frac{v(F_n,p)}2 \right)  \\
&=
\sum_{p\in D} \left(1 - \frac{v(F,p)}2 \right) \\
&=
\sum_{p\in K} \left(1 - \frac{v(F,p)}2 \right),
\end{align*}
for all sufficiently large $n$.
\end{proof}

\begin{definition}
Suppose that $F:M\to\RR$ is a Rad\'o function.  We say that $F$ is a {\bf minimal Rad\'o function}
provided $F$ has no interior local minima and no interior local maxima.
\end{definition}

We call these functions minimal Rad\'o functions because the Rad\'o functions that arise from minimal surfaces
as in Theorem~\ref{minimal-surface-theorem} have no interior local minima or interior local maxima.

\begin{theorem}\label{lower-semicontinuity-theorem}[Lower Semicontinuity Theorem]
Suppose that $F_n:M_n\to \RR$ and $F:M\to \RR$ are  tame minimal Rad\'o functions, where $M_n$ is an exhaustion
of $M$. Suppose also  that $\Tan(F_n,\cdot)$ converges locally uniformly to $\Tan(F,\cdot)$.
Then
\[
 \sum_{p\in M\setminus \partial M} w(F, p) \le \liminf  \sum_{p\in M_n\setminus \partial M_n} w(F_n, p).
\]
\end{theorem}

\begin{proof}
It suffices to prove it for $M$ without boundary. (Otherwise replace $M$ by $M\setminus \partial M$ in the following proof.)

Note that since there are no local maxima or minima, $w(F,p)$ and $w(F_n,p)$ are both nonnegative for every $p$.

Let $K$ be any compact subset of $M$.  Let $K'$ be a compact subset of $M$ such that $K\subset K'$ and such that 
$\partial K'$ lies in the regular set of $F$.
For all sufficiently large $n$,
\[
   \sum_{p\in M_n}w(F_n,p) \ge \sum_{p\in K'}w(F_n,p) = \sum_{p\in K'}w(F,p) \ge \sum_{p\in K}w(F,p),
\]
by Theorem~\ref{continuity-theorem}.  Thus
\[
  \liminf_n \sum_{p\in M}w(F_n,p) \ge \sum_{p\in K}w(F,p).
\]
Now take the supremum over all $K\subset\subset M$.
\end{proof}

\begin{corollary}\label{lower-semicontinuity-corollary}
Suppose that
\begin{enumerate}
\item $g_n$ are smooth Riemannian metrics on a $3$-manifold $N$ that converge smoothly to a metric $g$.
\item $\Ff_n$ are $g_n$-minimal foliations of $N$ that converge to a $g$-minimal foliation $\Ff$.
\item $M_n$ are $g_n$-minimal surfaces that converge smoothly to a $g$-minimal surface $M$.
\item No connected component of $M$ lies in a leaf of $\Ff$.
\end{enumerate}
Then
\begin{equation}\label{N-lower-semicontinuity}
  \mathsf{N}(\Ff,M) \le \liminf \mathsf{N}(\Ff_n,M_n)
\end{equation}
where $\mathsf{N}(\Ff,M)$ is the number of interior points of tangency of $\Ff$ and $M$, counting multiplicity.
\end{corollary}

In other words, $\mathsf{N}(\Ff,M)$ is the sum over the interior points $p\in M$ of the order of contact at $p$ of $M$ and the leaf of $\Ff$ through $p$.

\begin{proof}
Suppose first that $\Ff_n$ and $\Ff$ are given as the level sets of functions $F_n$ and $F$ on $N$
 as in Theorem~\ref{minimal-surface-theorem}.
Then $F_n|M_n$ and $F|M$ are tame Rad\'o functions by Theorems~\ref{minimal-surface-theorem}
and~\ref{minimal-tame-theorem}, and $\Tan(F_n|M_n,\cdot)$ converges locally uniformly
to $\Tan(F|M,\cdot)$ by Theorem~\ref{minimal-tame-theorem}.
Thus~\eqref{N-lower-semicontinuity} holds by Theorem~\ref{lower-semicontinuity-theorem}.

In general, $\Ff_n$ and/or $\Ff$ might not be expressible as the level sets of functions $F_n$ and $F$.  However, locally that is always
possible. Furthermore, even if $F_n$ and $F$ are only defined locally, $\Tan(F_n,\cdot)$ and $\Tan(F,\cdot)$ make
sense globally, and the proof of Theorem~\ref{lower-semicontinuity-theorem} only really depends on those line fields, and not
on the functions themselves.  Thus Corollary~\ref{lower-semicontinuity-corollary} holds for arbitrary minimal foliations.
\end{proof}

\section{Slices}

\newcommand{\set}{B}

A classical result of Rad\'o states that if a plane in $\RR^3$ intersects the boundary of a minimal disk in fewer than four points, then
it intersects the disk transversely.  In this section, we prove a very general form of Rad\'o's principle.

(When we apply the following theorem to minimal surfaces, the set $S$ and the set
of interior local maxima and minima will typically be empty.)

\begin{theorem}[Slice Theorem]\label{slice-theorem}
Let $M$ be a $2$-manifold 
and $F:M\to \RR$ be a continuous
function with only finitely many interior local minima and maxima.
Suppose that there is a finite set $S$ of interior points such that $F$ is a Rad\'o function on $M\setminus (\partial M\cup S)$.
Let $X=F^{-1}(t)$ be a level set of $M$.  
Suppose that
\begin{enumerate}[\upshape(1)]
\item $X$ is compact.
\item $X\cap \partial M$ is finite.
\item $d_1(M):=\dim H_1(M;\ZZ_2) <\infty$.
\end{enumerate}
Then $X$ is a finite network. 
 
Now suppose that $F$ is a Rad\'o function on all of $M\setminus \partial M$ (i.e., that the set $S$ is empty.)
Let 
\[
  V_n = \{p\in X: v(F,p)=n\}
\]
be the set of nodes of valence $n$ in $X$. 
Then
\begin{equation}\label{cadabra}
\begin{aligned}
\frac12 \sum_{n\ge 3} ( n- 2)  |V_n|  + d_0(X \setminus V_0)
&\le
\frac12 \, |V_1| +  d_1(M) +  |S^*| 
\\
&\le 
\frac12\, |J| + d_1(M) + |S^*|, 
\\
&\le
\frac12 \, k + d_1(M) + |S^*|,
\end{aligned}
\end{equation}
where $J$ is the set of points in $X\cap \partial M$ that are neither local maxima nor local minima of $F|\partial M$,
$S^*$ is the set of interior local maxima and local minima of $F$ in $M\cap\{F\ne t\}$, 
and $k$ is the number of points $p$ in $\partial M$ where $F(p)=t$.
\end{theorem}

Note that $d_0(X\setminus V_0)$ is the number of components of $X$ that are not isolated points.

In the examples that arise in minimal surface theory, $S^*$ is empty:

\begin{corollary}
If $S^*=\emptyset$,
then
\begin{align*}
\frac12 \sum_{n\ge 3}(n-2)\,|V_n|  + d_0(X \setminus V_0) 
&\le \frac12 |J| + d_1(M) \\
&\le \frac12 k + d_1(M),
\end{align*}
\end{corollary}

\begin{proof}[Proof of Theorem~\ref{slice-theorem}]
By Lemma~\ref{monomorphism-lemma}, the inclusion of $X$ into $M\setminus S^*$ induces a monomorphism of $H_1(\--; \ZZ_2)$, 
so 
\[
   d_1(X) \le d_1(M\setminus S^*) = d_1(M) + |S^*| < \infty.
\]
Thus, by a general theorem about networks (Theorem~\ref{finiteness-theorem}),
$X$ is a finite network.
(One lets the set $\set$ in Theorem~\ref{finiteness-theorem} be the union of the following
three sets: the set of interior local minima and interior local maxima of $F$ in $X$,
the set $S$, and the set $X \cap \partial M$.)
By a general counting theorem (Theorem~\ref{counting-theorem}) for finite graphs,
\begin{equation}\label{abra}
\begin{aligned}
\sum_{n\ge 3} \frac12 (n-2)|V_n| + d_0(X\setminus V_0)
&=
\frac12 |V_1|  + d_1(X) 
\\
&\le
\frac12 |V_1| + d_1(M) + |S^*|.
\end{aligned}
\end{equation}

Now suppose that $S$ is empty.
Recall that for every interior point $p$, $v(F,p)$ is even.
For a boundary point $p$, $V(F,p)$ is even if and only if $p$ is a local maximum or local minimum of $F|\partial M$.
Hence $V_1\subset J$.  Thus~\eqref{cadabra} follows from~\eqref{abra}.
\end{proof}

\begin{remark}
In the bound~\eqref{cadabra}, we could let $S^*$ be the set consisting of interior local minima of  $F$ in $\{F<t\}$ and of interior local maxima of $F$ in $\{F>t\}$.  No changes are required in the proof.  In the case of minimal Rad\'o functions, there are no interior local maxima or minima.
\end{remark}

We now prove the finiteness theorem for general networks that was used  in
the proof of Theorem~\ref{slice-theorem} (the Slice Theorem).

Let $p$ be a point in a topological space $X$ and $k$ be a nonnegative integer.
Suppose
$p$ has a neighborhood $U$ such that $X\cap U$ is the union of $k$ embedded curves,
where each curve joins $p$ to a point in $\partial U$ and where the curves intersect each other only at $p$.
Then we say that $X$ has {\bf valence} $k$ at $p$ and write $v(p)=V(X,p)=k$.
If there is no such $k$ and $U$, then $v(X,p)$ is undefined.

\begin{theorem}[Finiteness Theorem]\label{finiteness-theorem}
Let $X$ be a compact Hausdorff space.
Suppose that $\set\subset X$ is a finite set with the following properties.
\begin{enumerate}[\upshape (1)]
\item Each point $p$ in $X\setminus \set$ has a well-defined valence $v(p)$ that is $\ge 2$.
\item $d_1(X):=\dim H_1(X;\ZZ_2)$ is finite.
\end{enumerate}
Then $X$ is a finite network.
\end{theorem}

\begin{proof}
Define the set $V=V(X,\set)$ of vertices by
\[
V:= \set \cup \{p\in X\setminus \set : v(p)\ne 2\},
\]
and let $\Ee=\Ee(X,Q)$ be the set of connected components of $X\setminus V$.
Note that each element $E$ of $\Ee$ is an embedded curve and that $\overline{E}\setminus E\subset V$.
Elements of $\Ee$ are called {\bf edges}.  The assertion of the theorem is that $V$ and $\Ee$ are finite sets. 
We prove the theorem by induction on $d_1(X)$.

Suppose first that $d_1(X)=0$.  Let $T$ be a connected component of $X\setminus \set$.
Then $T$ is a connected network with no closed loops.  Thus $T$ is a tree.
Since $d_1(X)=0$, $d_1(\overline{T})=0$, and thus no two ends of $T$ can limit to the same point in $\set$.
Thus $T$ has at most $|\set|$ ends.  It follows that $\overline{T}$ is a finite tree.
Now $\overline{T}\setminus T$ is a subset of the finite set $\set$.
Thus if $X\setminus \set$ had infinitely many components, then it would have two components $T$ and $T'$ with
\[
  \overline{T}\setminus T = \overline{T'}\setminus T'.
\]
But then $T\cup T'$ would contain a loop, which is impossible since $d_1(X)=0$.
This completes the proof in the case $d_1(X)=0$.

Now suppose that $d_1(X)>0$.
Then $X$ contains a closed loop.  Let $E\in \Ee$ be an edge in that loop. Then
\[
  d_1(X\setminus E)< d_1(X).
\]
Let
\begin{align*}
\set'&=\set \cup \partial E, \\
X' &= X\setminus E. \\
\end{align*}
Then $X'$ and $\set'$ satisfy the hypotheses of the theorem and $d_1(X')<d_1(X)$, so (by induction),
$V(X',\set')$ and $\Ee(X',\set')$ are finite.  Consequently $V(X,\set)$ and $\Ee(X,\set)$ are also finite.
\end{proof}

The following counting theorem for arbitrary finite networks was used in the proof
of Theorem~\ref{slice-theorem} (The Slice Theorem).

\begin{theorem}[Counting Theorem]\label{counting-theorem}
Let $X$ be a finite network, and $V_n$ be the set
of vertices of valence $n$.  Then
\[
    \sum_{n\ge 3} \frac12(n-2) |V_n|  = \frac12 |V_1| +  d_1(X) -  d_0(X\setminus V_0),
\]
where $d_i(\cdot):=\dim H_i(\cdot;\ZZ_2)$.
\end{theorem}

(Note that $d_0(X\setminus V_0)$ is the number of connected components of $X$
that are not isolated points.)

\begin{proof}
Note that
\[
  \chi(X) = \sum_n \frac12(2-n) \,|V_n| = |V_0| + \frac12|V_1| + \frac12 \sum_{n\ge 3} (2-n)\,|V_n|
\]
(by Lemma~\ref{euler-lemma})
and 
\[
  \chi(X) = d_0(X) - d_1(X) = d_0(X\setminus V_0) + |V_0| - d_1(X).
\]
The assertion follows immediately.
\end{proof}

The Slice Theorem (Theorem~\ref{slice-theorem})
has the following important consequence:

\begin{theorem}\label{boundary-is-rado-theorem}
Suppose that $I\subset \RR$ is an open interval (possibly all of $\RR$) and that
 $F:M\to I$ is a proper continuous function such that $F$ is Rad\'o on the interior of $M$.
Suppose also that 
\begin{enumerate}
\item $d_1(M)<\infty$, 
\item for each $t$, $(\partial M)\cap\{F=t\}$ is finite, 
\item the set of interior local maxima and interior local minima of $F$ is finite.
\end{enumerate}
Then $F$ is a Rad\'o function on all of $M$.
\end{theorem}

The Slice Theorem also implies an interesting removal-of-singularities theorem:

\begin{theorem}\label{removal-theorem}
Suppose that $I\subset \RR$ is an open interval (possibly all of $\RR$),  that
 $F:M\to I$ is a proper continuous function, and that $S\subset M\setminus \partial M$ is a finite set
such that $F$ is Rad\'o on $M\setminus (S\cup \partial M)$.
Suppose also that 
\begin{enumerate}
\item $d_1(M)<\infty$, 
\item for each $t$, $(\partial M)\cap\{F=t\}$ is finite, 
\item the set of interior local maxima and interior local minima of $F$ is finite.
\end{enumerate}
Then $F$ is a Rad\'o function on all of $M$.
\end{theorem}

In minimal surface theory, one sometimes encounters functions that are Rad\'o
on the interior of the surface, but that are constant on some arcs and/or some connected components of the boundary.
No such function can be Rad\'o on the whole surface.  The following theorem lets one get around
that difficulty in many situations.

\begin{theorem}\label{interval-theorem}
Suppose that $I\subset \RR$ is an open interval (possibly all of $\RR$) and that
   $F:M\to \RR$ is a proper function that is Rad\'o on the interior of $M$.
Suppose also that
\begin{enumerate}
\item $d_1(M)<\infty$.
\item The set of interior local minima and interior local maxima of $F$ is finite.
\item There are only finitely many connected components of $\partial M$ on which $F$ is constant.
\item For each $t$, $(\partial M)\cap F^{-1}(t)$ is the union of finitely many connected components.
\end{enumerate}
Define an equivalent relation $\sim$ on $M$ as follows: $p\sim q$ if and only if $p=q$ or $p$ and $q$
belong to a connected subset of $\partial M$ on which $F$ is constant.
Let $\tilde M$ be the 	quotient $M/\sim$ and let $\tilde F$ be the function on $\tilde M$ corresponding to $F$.

Then $\tilde F: \tilde M\to \RR$ is a proper Rad\'o function.
\end{theorem}

\begin{proof}
Let $\Gamma$ be a connected component of $(\partial M)\cap F^{-1}(t)$.
If $\Gamma$ is a closed curve, then $\Gamma$ becomes an interior point $p$ of $\tilde M$.
If $\Gamma$ is not a closed curve, then $\Gamma$ becomes a boundary point $p$ of $\tilde M$.  In either case,
  $\tilde F(p)=t$.

Now let $S$ be the set of interior points 
in $\tilde M$ that correspond to closed curves in $\partial M$ along which $F$ is constant.

Then $\tilde F: \tilde M\to I$ and $S$ satisfy the hypotheses of Theorem~\ref{removal-theorem},
so $\tilde F$ is a Rad\'o function.
\end{proof}

\section{Branched Minimal Surfaces}\label{branch-section}

\begin{theorem}\label{branch-theorem}
Suppose that $N$ is a smooth Riemannian manifold and that
$F:M\to \RR$ is a continuous function such that
\begin{enumerate}
\item Each level set is a smooth minimal surface if $\dim N=3$.
\item Each level set is totally geodesic if $\dim N>3$.
\item For each $t$, the level set $F^{-1}(t)$ is in the closure of $\{F>t\}$ and of $\{F<t\}$.
\end{enumerate}
Suppose that $M$ is a connected surface without boundary, that $u:M\to N$ is a branched minimal immersion,
and that $F\circ u$ is not constant.
Then $F\circ u$ is a Rad\'o function without local minima or local maxima.
\end{theorem}

Of course if $M$ is minimal surface with boundary in $N$, then we can apply Theorem~\ref{branch-theorem} to conclude
that $F$ is Rad\'o on the interior of $M$, and then we can conclude from Theorem~\ref{boundary-is-rado-theorem} that $F\circ u$ 
is Rad\'o on all of $M$, provided the hypotheses of Theorem~\ref{boundary-is-rado-theorem} are satisfied.

\begin{proof}[Proof for $n=3$]
Since the result is local, it suffices to consider the case 
\begin{align*}
&u: \BB\subset \CC \to \RR^3, \\
&u(0)=0,
\end{align*}
where $\RR^3$ is endowed with a Riemannian metric
$g$ such that $g_{ij}(0)=\delta_{ij}$ and $Dg_{ij}(0)=0$.
By rotating, we can assume that, after a non-conformal reparametrization,
\begin{equation}\label{branch-form}
   u(z) = (z^Q, f(z)) \in \CC\times \RR \cong \RR^3,
\end{equation}
for some positive integer $Q$, where $f(z)= O(|z|^{Q+1})$.
See \cite{micallef-white}*{Theorem~1.4}.  

Now let $\Sigma$ be the level set of $F$ passing through the point $0=u(0)$.
If $\Tan(\Sigma,0)$ is not the horizontal plane $\RR^2\times \{0\}$, then the desired behavior follows
easily from~\eqref{branch-form}.  Indeed, in this case, minimality of $\Sigma$ is not even needed.  
Note that in this case the valence of the point is $v(F\circ u, 0) = 2Q$ and thus $w(F\circ u, 0) = Q-1$.

Now suppose that $\Tan(\Sigma,0)$ is the horizontal plane $\RR^2\times \{0\}$.
Then, near the origin, $\Sigma$ is the graph of a function $\phi:\Omega\subset \RR^2\to \RR$
with $\phi(0)=0$ and $D\phi(0)=0$.

Note that
\[
   z \mapsto (z^Q, \phi(z^Q))
\]
is a nonconformal reparametrization of a branched minimal immersion. (If $Q>1$, then $0$ is a false branch point of order $Q-1$).

Now consider the map
\begin{equation}\label{the-difference}
   z\in D \mapsto f(z) - \phi(z^Q).
\end{equation}
If this were identically $0$, then $F\circ u$ would be constant on a neighorhood of $0$ and thus on all of $M$ by unique continuation,
contrary to the hypotheses of the theorem.
Thus the function~\eqref{the-difference} is not constant.  By~\cite{micallef-white}*{1.6},
 there is a nonzero homogenious polynomal $h$ of degree
 $d\ge Q$ such that
\begin{align*}
&f(z) - \phi(z^Q) = h(z) + o(|z|^d), \\
&D(f(z) - \phi(z^Q)) = Dh(z) + o (|z|^{d-1}).
\end{align*}
The desired behavior near $0$ of the level set of $F\circ u$ through $0$ follows immediately.
Note that in this case, $v(F\circ u,0)=2d \ge 2Q$ and hence $w(F\circ u, 0)\ge Q-1$.
\end{proof}

\begin{proof}[Proof for $n>3$]
The result is local, so we may assume that the branched immersion is
\[
   u: D\subset \RR^2 \to N.
\]
We wish to prove that the level set of $F\circ u$ through the origin has the behavior specified in the definition of Rad\'o function.
Let $\Sigma$ be the level set of $F$ through $p=u(0)$.
Choose Fermi-type local coordinates on $N$ as follows.
First, let $(y^1, \dots, y^{n-1})$ be normal coordinates on $\Sigma$ at $p$.
For $q$ in $N$ near $p$, let $y^n(q)$ be the signed distance from $q$ to $\Sigma$, and
for $i<n$, let $y^i(q)=y^i(q')$ where $q'$ is the point in $\Sigma$ closest to $q$.

Thus
\begin{align*}
g_{ni} &= \delta_{ni}  \quad (i\le n), \\
g_{ij}(0)&=\delta_{ij}, \\
Dg_{ij}(0) &= 0.
\end{align*}

Now $u$ is a conformal harmonic map.  Harmonicity means that
\[
   \pdf{}{x^k} \left( g_{i \alpha}(u(x)) \pdf{u^i}{x^k} \right) - (D_\alpha g_{ij}(u(x))  \pdf{u^i}{x^k} \pdf{u^j}{x^k}  = 0 
\]
for each $\alpha = 1, \dots, n$.
Here, $k$ is summed from $1$ to $2$ and the other repeated indices from $1$ to $n$.
In particular, this holds for $\alpha=n$:
\begin{equation}\label{totally-geodesic-equation}
\begin{aligned}
0 
&=
\pdf{}{x^k} \left( g_{in}(u(x)) \pdf{u^i}{x^k} \right) 
     - (D_n g_{ij}(u(x))  \pdf{u^i}{x^k} \pdf{u^j}{x^k}  \\
&=
\Delta u^n  + (D_n g_{ij}(u(x))  \pdf{u^i}{x^k} \pdf{u^j}{x^k}.
\end{aligned}
\end{equation}
since $g_{ni}\equiv \delta_{ni}$.

If $i$ and/or $j$ is $n$, then $g_{ij}$ is constant, so $D_n g_{ij}\equiv 0$.
On the other hand, if $i$ and $j$ are less than $n$, then 
\[
    D_n g_{ij}(y)=0  \quad \text{when $y^n=0$}
\]
since $\Sigma$ is totally geodesic, and therefore
\[
    |D_n g_{ij}(y)| \le c\, |y^n|
\]
for $|y|\le r$ and for some constant $c=c_r$.  Thus from~\eqref{totally-geodesic-equation}, we see that
\begin{equation}\label{laplacian-bound}
  |\Delta u^n| \le K |u^n|.
\end{equation}
By hypothesis, $F\circ u$ is not constant.  Thus by~\eqref{laplacian-bound} and the
 Hartman-Wintner Theorem (as formulated in~\cite{micallef-white}*{Theorem~1.1}),
  there is a nonzero homogeneous harmonic polynomial $h$
of degree $d\ge 1$ such that
\begin{align*}
    u^n(z) &= h(z) + o(|z|^d), \\
    Du^n(z) &= Dh(z) + o(|z|^{d-1}).
\end{align*}
The assertion follows immediately.

Note that the valence $v(F\circ u, 0)$ is $2d$ and thus $w(F\circ u, 0) = d-1$.
We remark that if $u$ has branch point of order $Q-1$ at the origin, then 
\[
   \lim_{z\to 0} \frac{|u(z)|}{|z|^Q}
\]
 exists and is nonzero.
Thus  $d\le Q$.  Note that $d > Q$ if only if $u^n = 0(|u|) = o(|z|^Q)$, 
i.e., if and only if the the tangent plane to $u(D)$ at $0$ is contained in $\Tan(\Sigma,u(0))$. 
\end{proof}

Here we summarize the facts about saddle-multiplicity and branch point order that were established in proving Theorem~\ref{branch-theorem}:

\begin{theorem}\label{branch-multiplicity-theorem}
Suppose that $u:M\to N$ and $F:N\to \RR$ are as in Theorem~\ref{branch-theorem}.
If $p$ is not a branch point, then $w(F\circ u, p)$ is the order of contact of $u(M)$ and $\{F=F(p)\}$
at $u(p)$
Now suppose that $p$ is a branch point of order $m$.
Then
\begin{equation}\label{saddle-branching-inequality}
   w(F\circ u, p) \ge m.
\end{equation}
Equality holds if the tangent plane to $u(M)$ at $u(p)$ and the tangent plane to $\{F=F(u(p))\}$ at $u(p)$
are transverse.

In case the level sets of $F$ are totally geodesic (as they are if $n>3$), equality holds in~\eqref{saddle-branching-inequality}
  if and only if the tangent planes are transverse.
\end{theorem}

\begin{corollary}\label{branch-tame-corollary}
The function $F\circ u$ is a tame Rad\'o function.
\end{corollary}

\begin{proof}
Tameness is a local property of the regular points of a Rad\'o function.
Since all branch points are critical points, tameness of $F\circ u$ follows
from the embedded case (Theorem~\ref{minimal-surface-theorem}).
\end{proof}

\appendix

\section{Some basic topological facts}

\begin{lemma}\label{annulus-lemma}
Suppose that $C_i$, $i=1, \dots, k$ are disjoint simple closed curves in an open annulus $U$ and that each $C_i$ is homotopically
nontrivial in $U$.  Then each component of $U\setminus (\cup_i C_i)$ is an annulus.
\end{lemma}

The proof is a simple induction.

Recall that if $F:\Sigma\to \RR$ and if $s< t$,
we let $\Sigma[s] = F^{-1}(s)$ and $\Sigma[s,t] = F^{-1}([s,t])$.
\[
\]

\begin{proposition}\label{annulus-proposition}
Let $\Sigma$ be a compact, connected $2$-manifold with boundary and let
\[
  F:\Sigma\to [a,b]
\]
be a continuous function such that 
\begin{enumerate}
\item Each level set $\Sigma[t]$ is a finite union of disjoint simple closed curves,
\item  $F$ has no interior local maxima or minima,
\item $\partial \Sigma = \Sigma[a] \cup \Sigma[b]$.
\end{enumerate}
Then $\Sigma$ is  an annulus.
\end{proposition}

\begin{proof}
Note that 
\begin{equation}\label{no-region}
\text{No collection of curves in $\Sigma[t]$ can bound a region in $\Sigma$.} \tag{*}
\end{equation}
For if there were such a region $D$, then $F|\overline{D}$ would attain its maximum and/or its minimum at an interior point of $D$, violating~[2].

\begin{claim}\label{annulus-claim}
Let $a\le s< t\le b$.  Let $K$ be a connected component of $\Sigma[s,t]$.
If $K$ lies in an annular region $U$ of $\Sigma$, then $K$ is an annulus, with one boundary component in $\Sigma[s]$ and one
in $\Sigma[t]$.
\end{claim}

\begin{proof}
Note that $\partial K$ is contained in $\Sigma[s]\cup \Sigma[t]$.
By~\eqref{no-region}, it must have at least one component in $\Sigma[s]$ and at least one component in $\Sigma[t]$.
If $C$ is a component of $\partial K$, it cannot bound a region in $\Sigma$ by~\eqref{no-region}. 
In particular, it does not bound a disk in $U$. 
Thus each component of $\partial K$ is homotopically nontrivial in $U$.  By Lemma~\ref{annulus-lemma}, $K$ is an annulus.
Thus we have proved Claim~\ref{annulus-claim}.
\end{proof}

Note for each $T\in [a,b]$, there is a relatively open subset $U_T$ of $\Sigma$ containing $\Sigma[T]$
such that $U$ contains $\Sigma[T]$ and such that each component of $U$ is an annulus.
Now $F(\Sigma\setminus U)$ is a compact subset of $[a,b]$ that does not contain $T$.
Thus there is an open interval $I_T\subset \RR$ containing $T$ and disjoint from $F(\Sigma\setminus U_T)$.

Since the $\{I_T: T\in [a,b]\}$ form an open cover of $I$, there exists $a_0=a< a_1 < a_2 < \dots < a_n=b$
such that each $[a_{i-1},a_i]$ belongs to some $I_T$.

By Claim~\ref{annulus-claim}, each component of $\Sigma[a_{i-1},a_i]$ is an annulus with one component in $\Sigma[a_{i-1}]$
and one component in $\Sigma[a_i]$.   Proposition~\ref{annulus-proposition} follows immediately.
\end{proof}

\begin{corollary}\label{annulus-corollary}
Suppose that 
\begin{enumerate}
\item $\Sigma$ is an open surface.
\item $F:\Sigma\to (a,b)$ is a continuous function with no local maxima or local minima.
\item For each $t\in (a,b)$, $\Sigma[t]$ is a union of finitely many disjoint simple closed curves.
\item If $a<s< t<b$, then $\Sigma[s,t]$ is compact and 
\[
  \partial \Sigma[s,t] = \Sigma[s] \cup \Sigma[t].
\]
\end{enumerate}
Then each connected component of $\Sigma$ is an annulus.
\end{corollary}

\begin{proof}
Let $a_i, \, i\in \ZZ$ be a strictly increasing sequence with $\lim_{i\to -\infty}a_i=a$ and $\lim_{i\to\infty}a_i=b$.
By Proposition~\ref{annulus-proposition}, each component of $\Sigma[a_{i-1}, a_i]$ is an annulus with one component in $\Sigma[a_{i-1}]$
and one component in $\Sigma[a_i]$. Corollary~\ref{annulus-corollary} follows immediately.
\end{proof}

Let $n\ge 3$.
We define an {\bf $n$-ad} $K$ to be a closed set consisting of $(n+1)$ points $p_0, p_1, \dots, p_n$, together with $n$ embedded arcs 
 $\gamma_1, \dots, \gamma_n$, where each $\gamma_i$ joins $p_0$ to $p_i$, and where $\gamma_i\cap\gamma_j=\{p_0\}$
 for $i\ne j$.  
We say that $p_0$ is the {\bf center} of the $n$-ad,
and that the points $p_1,\dots, p_n$ are the {\bf boundary} $\partial K$ of the $n$-ad.

\begin{lemma}\label{n-ad-lemma}
Let $M$ be a $2$-manifold and let $n\ge 3$.
Suppose $C$ is a collection of disjoint subsets of $M$, each of which contains an $n$-ad.
Then $C$ is countable.
\end{lemma}

\begin{proof}
Suppose to the contrary that there is an uncountable collection $C$.
We can assume that each set in $C$ is an $n$-ad. (Otherwise, replace each set in $C$ by an $n$-ad that it contains.)
Consider a countable collection $\Dd$ of open disks $D\subset M$ such that $\overline{D}$ is a closed disk and such
that $\Dd$ is a basis for the topology of $M$.
For $D\in \Dd$, let $C_D$ be the collection of $K\in C$ such that the center of $K$ is in $D$ and such that the boundary points
of $K$ are not in $D$.  Note that there is a $D\in \Dd$ for which $C_D$ is uncountable.  For each $K\in C_D$, let $K'$ be the closure
of $K\cap D$.  Then $K'$ is an $n$-ad with center in $D$ and with 
boundary in $\partial D$.  Let
\[
  C_D' = \{K': K \in C_D\}.
\]
Then $C_D'$ is uncountable.  Define $\Phi: C_D'\to  (\partial D)^n$ by
\[
  \Phi(K) = (p_1, \dots, p_n),
\]
where $p_1, \dots, p_n$ are the boundary points of $K$.  (We choose an ordering of the endpoints.)  Since $C_D'$ is uncountable and since $(\partial D)^n$ is separable, there exist $K_i$ ($i\in N$) and $K$
in $C_D'$ such that $\Phi(K_i)\to \Phi(K)$.
But that is impossible since each $K_i$ lies in one of the connected components of $\overline{D}\setminus K$.

(If the last sentence is not clear, note that $\partial K_i$ must lie in an arc $A$ of $(\partial D)\setminus \partial K$.
Since $\partial K$ has $n\ge 3$ points, it has a point $p$ that is not in $\overline{A}$.  Since $\partial K_i\subset A$, the points
of $\partial K_i$ are bounded away from $p$.)
\end{proof}

\begin{corollary}\label{n-ad-corollary}
Suppose $F:M\to \RR$ is a Rado function. Then there are only countably many critical points, and hence only countably many critical 
values.
\end{corollary}

\begin{proof}
By doubling, it suffices to prove it for $M$ without boundary.  The points of valence $0$ are strict local maxima or minima,
and hence there are only countably many of them.  The other critical points are points of valence $\ge 4$.   
By Lemma~\ref{n-ad-lemma}, for each $v\ge 4$, there are only countably many  points of valence $v$.
\end{proof}



\begin{bibdiv}
\begin{biblist}

\bib{colding-minicozzi}{book}{
   author={Colding, Tobias Holck},
   author={Minicozzi, William P., II},
   title={A course in minimal surfaces},
   series={Graduate Studies in Mathematics},
   volume={121},
   publisher={American Mathematical Society, Providence, RI},
   date={2011},
   pages={xii+313},
   isbn={978-0-8218-5323-8},
   review={\MR{2780140}},
   doi={10.1090/gsm/121},
}

\bib{FO}{article}{
   author={Finn, R.},
   author={Osserman, R.},
   title={On the Gauss curvature of non-parametric minimal surfaces},
   journal={J. Analyse Math.},
   volume={12},
   date={1964},
   pages={351--364},
   issn={0021-7670},
   review={\MR{166694}},
   doi={10.1007/BF02807440},
}
		
\bib{?}{article}{
   author={Gama, E. S.},
      author={Mart\'{\i}n, F.},
   author={M\o{}ller, N. M.},
   title={Finite entropy translating solitons in slabs},
   journal={Preprint arXiv:2209.01640},
    date={2022},
}

\bib{graphs}{article}{
   author={Hoffman, D.},
   author={Ilmanen, T.},
   author={Mart\'{\i}n, F.},
   author={White, B.},
   title={Graphical translators for mean curvature flow},
   journal={Calc. Var. Partial Differential Equations},
   volume={58},
   date={2019},
   number={4},
   pages={Paper No. 117, 29},
   issn={0944-2669},
   review={\MR{3962912}},
   doi={10.1007/s00526-019-1560-x},
}

\bib{himw-correction}{article}{
   author={Hoffman, D.},
   author={Ilmanen, T.},
   author={Mart\'{\i}n, F.},
   author={White, B.},
   title={Correction to: Graphical translators for mean curvature flow},
   journal={Calc. Var. Partial Differential Equations},
   volume={58},
   date={2019},
   number={4},
   pages={Art. 158, 1},
   issn={0944-2669},
   review={\MR{4029723}},
   review={Zbl 07091751},
   doi={10.1007/s00526-019-1601-5},
}
		

	


\bib{HMW1}{article}{
author={Hoffman, D.},
   author={Mart\'{\i}n, F.},
   author={White, B.},
   title={Scherk-like translators for mean curvature flow},
   journal={J. Differential Geom.},
   volume={122},
   date={2022},
   number={3},
   pages={421--465},
   issn={0022-040X},
   review={\MR{4544559}},
   doi={10.4310/jdg/1675712995},
}
\bib{HMW2}{article}{
author={Hoffman, D.},
author={Martín, F.},
author={White, B.},
title={Nguyen's tridents and the classification of semigraphical translators for mean curvature flow},
date={2022},
journal={Journal für die reine und angewandte Mathematik (Crelles Journal)},
volume={2022},
number={786},
pages={79--105},
}

\bib{HMW3}{article}{
author={Hoffman, D.},
author={Martín, F.},
author={White, B.},
title={Translating annuli for Mean Curvature Flow.},
date={2022},
journal={In preparation},
}

\bib{hoffman-white}{article}{
   author={Hoffman, David},
   author={White, Brian},
   title={The geometry of genus-one helicoids},
   journal={Comment. Math. Helv.},
   volume={84},
   date={2009},
   number={3},
   pages={547--569},
   issn={0010-2571},
   review={\MR{2507253}},
   doi={10.4171/CMH/172},
}

\bib{maxwell}{article}{
author = {Maxwell, J. Clerk},
title = {On hills and dales},
journal = {The London, Edinburgh, and Dublin Philosophical Magazine and Journal of Science},
volume = {40},
number = {269},
pages = {421--427},
year  = {1870},
publisher = {Taylor & Francis},
doi = {10.1080/14786447008640422},
URL = { 
        https://doi.org/10.1080/14786447008640422   
},
eprint = {https://doi.org/10.1080/14786447008640422},
}        
 


\bib{micallef-white}{article}{
   author={Micallef, Mario J.},
   author={White, Brian},
   title={The structure of branch points in minimal surfaces and in
   pseudoholomorphic curves},
   journal={Ann. of Math. (2)},
   volume={141},
   date={1995},
   number={1},
   pages={35--85},
   issn={0003-486X},
   review={\MR{1314031}},
   doi={10.2307/2118627},
}


\bib{osserman-book}{book}{
   author={Osserman, Robert},
   title={A survey of minimal surfaces},
   edition={2},
   publisher={Dover Publications, Inc., New York},
   date={1986},
   pages={vi+207},
   isbn={0-486-64998-9},
   review={\MR{852409}},
}

\bib{Rado1930}{article}{
   author={Rad\'{o}, Tibor},
   title={The problem of the least area and the problem of Plateau},
   journal={Math. Z.},
   volume={32},
   date={1930},
   number={1},
   pages={763--796},
   issn={0025-5874},
   review={\MR{1545197}},
   doi={10.1007/BF01194665},
}

\bib{rado}{book}{
   author={Rad\'{o}, Tibor},
   title={On the Problem of Plateau},
   publisher={Chelsea Publishing Co., New York, N. Y.},
   date={1951},
   pages={iv+109},
   review={\MR{0040601}},
}

\bib{schneider}{article}{
   author={Schneider, Rolf},
   title={A note on branch points of minimal surfaces},
   journal={Proc. Amer. Math. Soc.},
   volume={17},
   date={1966},
   pages={1254--1257},
   issn={0002-9939},
   review={\MR{202068}},
   doi={10.2307/2035720},
}

\bib{solomon}{article}{
   author={Solomon, Bruce},
   title={On foliations of ${\bf R}^{n+1}$ by minimal hypersurfaces},
   journal={Comment. Math. Helv.},
   volume={61},
   date={1986},
   number={1},
   pages={67--83},
   issn={0010-2571},
   review={\MR{847521}},
   doi={10.1007/BF02621903},
}


\end{biblist}
\end{bibdiv}
\end{document}